\documentclass[11pt]{amsart}

\usepackage[a4paper,hmargin=3.5cm,vmargin=3.9cm]{geometry}
\usepackage{amsfonts,amssymb,amscd,amstext}
\usepackage{graphicx}
\usepackage[dvips]{epsfig}
\usepackage{quoting}
\usepackage{hyperref}
\usepackage{verbatim}

\usepackage{fancyhdr}
\input xy
\xyoption{all}

\usepackage{enumerate}
\usepackage{titlesec}
\usepackage{mathrsfs}

\titleformat{\section}
{\filcenter\bfseries} {\thesection{.}}{0.2cm}{}

\titleformat{\subsection}[runin]
{\bfseries} {\thesubsection{.}}{0.15cm}{}[.]

\titleformat{\subsubsection}[runin]
{\em}{\thesubsubsection{.}}{0.15cm}{}[.]

\usepackage[up,bf]{caption}

\newtheorem{theorem}{Theorem}[section]
\newtheorem{proposition}[theorem]{Proposition}

\newtheorem{corollary}[theorem]{Corollary}

\theoremstyle{definition}
\newtheorem{definition}[theorem]{Definition}
\newtheorem{remark}[theorem]{Remark}

\numberwithin{equation}{section}
\numberwithin{figure}{section}

\newcommand\Acal{\mathcal{A}}

\newcommand\Fcal{\mathcal{F}}

\newcommand\Pcal{\mathcal{P}}

\newcommand\Ascr{\mathscr{A}}

\newcommand\Cscr{\mathscr{C}}

\newcommand\Iscr{\mathscr{I}}

\newcommand\Oscr{\mathscr{O}}

\newcommand\C{\mathbb{C}}

\newcommand\N{\mathbb{N}}
\renewcommand\P{\mathbb{P}}

\newcommand\R{\mathbb{R}}

\newcommand\Z{\mathbb{Z}}

\newcommand\hgot{\mathfrak{h}}
\newcommand\igot{\mathfrak{i}}

\renewcommand\igot{\mathfrak{i}}

\renewcommand\imath{\igot}

\newcommand\hra{\hookrightarrow}

\newcommand\longhookrightarrow{\ensuremath{\lhook\joinrel\relbar\joinrel\rightarrow}}

\newcommand\wt{\widetilde}

\newcommand\di{\partial}
\newcommand\dibar{\overline\partial}

\newcommand\dist{\mathrm{dist}}

\newcommand\Flux{\mathrm{Flux}}

\newcommand\CMI{\mathrm{CMI}}
\newcommand\NC{\mathrm{NC}}

\def\dist{\mathrm{dist}}

\def\Flux{\mathrm{Flux}}

\newcommand\boldA{\mathbf A}

\usepackage[color=blue!20]{todonotes}

\usepackage{color}

\begin{document}


\fancyhead[LO]{Isotopies of complete minimal surfaces of finite total curvature}
\fancyhead[RE]{A.\ Alarc\'on, F.\ Forstneri\v c and F.\ L\'arusson}
\fancyhead[RO,LE]{\thepage}

\thispagestyle{empty}



\begin{center}
{\bf \Large Isotopies of complete minimal surfaces
\\ \vspace{1mm}  
of finite total curvature} 

\bigskip

%
%
{\bf Antonio Alarc\'on, Franc Forstneri\v c, and Finnur L\'arusson}
\end{center}


%
%

\bigskip

\begin{quoting}[leftmargin={5mm}]
{\small
\noindent {\bf Abstract}\hspace*{0.1cm}  
Let $M$ be a Riemann surface biholomorphic to an affine algebraic curve.
We show that the inclusion of the space 
$\Re \NC_*(M,\C^n)$ of real parts of nonflat proper algebraic null immersions
$M\to\C^n$, $n\ge 3$, into the space $\CMI_*(M,\R^n)$ of 
complete nonflat conformal minimal immersions 
$M\to\R^n$ of finite total curvature is a weak homotopy equivalence. 
We also show that the $(1,0)$-differential $\di$, 
mapping $\CMI_*(M,\R^n)$ or  $\Re \NC_*(M,\C^n)$
to the space $\Ascr^1(M,\boldA)$ of algebraic $1$-forms on $M$ 
with values in the punctured null quadric $\boldA\subset \C^n\setminus\{0\}$, 
is a weak homotopy equivalence. 
Analogous results are obtained for proper algebraic immersions $M\to\C^n$,
$n\ge 2$, directed by a flexible or algebraically elliptic punctured
cone in $\C^n\setminus\{0\}$.

\smallskip

\noindent{\bf Keywords}\hspace*{0.1cm} 
Riemann surface, minimal surface, algebraic immersion, 
directed immersion, algebraically elliptic manifold, flexible manifold,
Oka manifold

\noindent{\bf Mathematics Subject Classification (2020)}\hspace*{0.1cm} 
Primary 53A10. Secondary 30F99, 32E30, 32H02, 32Q56, 53C42

\noindent{\bf Date}\hspace*{0.1cm} 5 June 2024
}

\end{quoting}

%
%
\section{Introduction}\label{sec:intro}

\noindent
In this paper, we initiate the study of the homotopy theory of minimal surfaces 
of finite total curvature in Euclidean spaces $\R^n$, $n\ge 3$,
and of algebraic directed immersions from affine algebraic
curves to complex Euclidean spaces $\C^n$, $n\geq 2$. 

We begin by recalling the basic facts on minimal surfaces, 
referring to \cite{Osserman1986} and \cite{AlarconForstnericLopez2021} 
for more information. Let $M$ be an open Riemann surface. An immersion 
$u=(u_1,\ldots,u_n):M\to\R^n$ is conformal if and only if the $\C^n$-valued  
$(1,0)$-form $\phi=(\phi_1,\ldots,\phi_n)$, whose $i$-th component
$\phi_i=\di u_i$ is the $(1,0)$-differential of the function $u_i$,
satisfies the nullity condition
\begin{equation}\label{eq:nullity}
	\sum_{i=1}^n \phi_i^2=0.
\end{equation}
A conformal immersion $u:M\to\R^n$ parametrises a minimal surface
if and only if $u$ is harmonic if and only if
$\phi=\di u$ is a holomorphic 1-form. Conversely, 
a nowhere vanishing holomorphic 1-form $\phi=(\phi_1,\ldots,\phi_n)$
on $M$ satisfying the nullity condition \eqref{eq:nullity} and the period 
vanishing condition $\Re \int_C \phi =0$ for all closed curves $C$ in $M$
integrates to a conformal minimal immersion $u=\Re \int \phi:M\to\R^n$. 
If $\int_C \phi =0$ for all closed curves $C$ in $M$,
then $\phi$ integrates to a holomorphic null curve $h=\int \phi:M\to\C^n$
whose real and imaginary parts parametrise conjugate minimal surfaces. 
Define the null quadric 
\begin{equation}\label{eq:nullq}
	\boldA =\bigl\{(z_1,\ldots,z_n)\in \C^n_* =\C^n\setminus \{0\}: 
	z_1^2+z_2^2+\cdots + z_n^2=0\bigr\}.
\end{equation}
Given a Riemann surface $M$, the null quadric 
defines a holomorphic subbundle $\mathcal A$ of the vector bundle 
$(T^*M)^{\oplus n}$, with fibre isomorphic to $\boldA$, whose sections 
are $n$-tuples $(\phi_1,\ldots,\phi_n)$ 
of $(1,0)$-forms on $M$ without common zeros such that the ratio 
$[\phi_1:\cdots:\phi_n]:M\to \P^{n-1}$ takes values in the projective quadric
defined by the same equation $z_1^2+z_2^2+\cdots + z_n^2=0$.
Hence, a smooth map $u:M\to\R^n$ is a conformal minimal immersion if and 
only if $\phi=\di u$ is a holomorphic section of the bundle $\Acal\to M$.
Note that the fibre multiplication by a nonzero complex 
number is a well-defined operation on $\Acal$. 

The flux homomorphism $\Flux_u : H_1(M,\Z)\to \R^n$
of a conformal minimal immersion $u:M\to\R^n$ is given by
\begin{equation}\label{eq:FluxuC}
	\Flux_u(C)= \int_C d^c u 
	\quad \text{for every}\ \ [C]\in H_1(M,\Z).
\end{equation}
We may view $\Flux_u$ as an element of 
the cohomology group $H^1(M,\R^n)$. 
The $1$-form $d^c u=\Im(2\di u)=\imath(\dibar u-\di u)$ 
is the differential of any local harmonic conjugate of $u$, so  
$u$ is the real part of a holomorphic null curve 
$M\to\C^n$ if and only if $\Flux_u=0$.

Complete minimal surfaces of finite total Gaussian curvature 
are among the most intensively studied
minimal surfaces, and they play an important role in the 
classical global theory of minimal surfaces;
see \cite{Osserman1986} and \cite[Chapter 4]{AlarconForstnericLopez2021},
among many other sources. If $u:M\to\R^n$ is 
a conformally immersed minimal surface of finite total curvature, 
then the Riemann surface $M$ is the complement in a compact 
Riemann surface $\overline M$ of a nonempty finite set 
$E=\{x_1, \ldots, x_m\}$ whose points are called the ends of $M$. 
Such a surface $M$ admits a biholomorphism 
onto a closed embedded algebraic curve in $\C^3$ and  
hence will be called an {\em affine Riemann surface}. 
The bundle $\Acal \to M$ with fibre $\boldA$ \eqref{eq:nullq} is algebraic 
and $\di u$ is a meromorphic $1$-form on 
$\overline M$ without zeros or poles on $M$, that is, an algebraic (regular)
section of $\Acal$ over $M$.  Completeness of $u$ is reflected in
$\di u$ having an effective pole at every point of $E=\overline M\setminus M$.
The surface $u(M)$ is then properly immersed in $\R^n$ 
and has a fairly simple and 
well-understood asymptotic behaviour at every end
of $M$, described by the Jorge--Meeks theorem \cite{JorgeMeeks1983T}.
Although this family of minimal surfaces has been a focus of interest 
since the seminal work of Osserman in the 1960s \cite{Osserman1986}, 
the theories of approximation and interpolation for complete minimal surfaces 
of finite total curvature in $\mathbb{R}^n$, including Runge and 
Mittag--Leffler type theorems, have been developed only recently 
\cite{Lopez2014TAMS,AlarconCastro-InfantesLopez2019CVPDE,AlarconLopez2022APDE}.  
The present paper provides the first contributions 
to the homotopy theory of this 
most important class of minimal surfaces.

We denote by $\CMI_*(M,\R^n)$ the space of complete nonflat conformal 
minimal immersions $M\to \R^n$ of finite total curvature, and by 
$\NC_*(M,\C^n)$ and $\Re\NC_*(M,\C^n)$ the spaces of complete 
nonflat holomorphic null immersions $M\to \C^n$ of finite total curvature
(that is, proper and algebraic) and their real parts, respectively. 
We denote by $\Ascr^1(M,\boldA)$ the space of $\C^n$-valued meromorphic 
$1$-forms $\phi=(\phi_1,\ldots,\phi_n)$ on $\overline M$ having no zeros or 
poles in $M$ and satisfying the nullity condition \eqref{eq:nullity}, 
that is, algebraic sections of the bundle $\mathcal A\to M$, and by 
$\Ascr^1_*(M,\boldA)$ the subspace of nonflat 1-forms. 
Nonflatness means that the map $[\phi_1:\cdots:\phi_n]:M\to \P^{n-1}$ 
is nonconstant. All these spaces are endowed with the compact-open
topology.  

Consider the diagram
\begin{equation}\label{eq:diagram}
\xymatrix{
	\Re \NC_*(M,\C^n)  \ar@{^{(}->}[r] \ar[dr]_\di  
	&  \CMI_*(M,\R^n) \ar[d]^\di  \\
	&   \Ascr^1(M,\boldA)
}
\end{equation}
where $\di$ is the $(1,0)$-differential. The following is our main result.

%
%
\begin{theorem}\label{th:whe1}
If $M$ is an affine Riemann surface, then the maps in \eqref{eq:diagram}
are weak homotopy equivalences. 
\end{theorem}

Recall that a continuous map $\alpha: X\to Y$ between topological spaces
is said to be a weak homotopy equivalence if it induces
a bijection of path components of the two spaces and an isomorphism
$\pi_k(\alpha):\pi_k(X)\to\pi_k(Y)$ of their homotopy groups for $k=1,2,\ldots$ 
and arbitrary base points.
Thus, Theorem \ref{th:whe1} says that the three mapping spaces in 
diagram \eqref{eq:diagram} have the same rough topological shape.

Note that the images of the maps $\di$ in \eqref{eq:diagram} are contained 
in the subspace $\Ascr^1_\infty(M,\boldA)$ of $\Ascr^1_*(M,\boldA)$ 
consisting of nonflat $1$-forms that have poles at all ends of $M$. 
It turns out that $\Ascr^1_\infty(M,\boldA)$ and $\Ascr^1_*(M,\boldA)$
are open everywhere dense subsets of $\Ascr^1(M,\boldA)$, and the inclusions
\begin{equation}\label{eq:inclusions}
	\Ascr^1_\infty(M,\boldA) \longhookrightarrow \Ascr^1_*(M,\boldA)
	\longhookrightarrow \Ascr^1(M,\boldA)
\end{equation}
are weak homotopy equivalences (see Propositions 
\ref{prop:poles} and \ref{prop:nondegenerate}).
Hence, Theorem \ref{th:whe1} also 
holds if the operator $\di$ is considered as a map to any of these two smaller 
spaces of $1$-forms on $M$. See the more detailed diagram 
\eqref{eq:diagram-extended} and Corollary \ref{cor:whe-extended}.	

Recall that the tangent bundle of an affine Riemann surface $M$ is 
holomorphically trivial (this holds for every open Riemann
surface by a theorem of Oka \cite{Oka1939}, 
see also \cite[Theorem 5.3.1~(c)]{Forstneric2017E} and
the more precise result of Gunning and Narasimhan 
\cite{GunningNarasimhan1967}), but is not
algebraically trivial in general.  If the genus of $M$ is 0 or 1, 
that is, $\overline M$ is the Riemann sphere $\P$ or a torus, 
then $TM$ is algebraically trivial. 
More generally, $TM$ is algebraically trivial if $M$ is 
embeddable as a closed algebraic curve in $\C^2$.
Note that $TM$ is algebraically trivial if and only if the cotangent bundle
$T^*M$ is such, and this holds if and only if there is an 
algebraic $1$-form $\theta$ on $M$ without zeros or poles. 
Such a 1-form is the restriction to $M$ of a meromorphic 
1-form on $\overline M$. Every $\phi\in \Ascr^1(M,\boldA)$ 
is then of the form $\phi=f\theta$ where $f:M\to \boldA$ is an algebraic 
map. In this case we can replace the space $\Ascr^1(M,\boldA)$
in \eqref{eq:diagram} by the space $\Ascr(M,\boldA)$ of 
algebraic maps $M\to\boldA$ and the differential $u\mapsto \di u$ 
by the map $u\mapsto \di u/\theta$.

For the spaces of nonflat holomorphic null curves and 
nonflat conformal minimal immersions 
(not necessarily complete or of finite total curvature) 
from an arbitrary open Riemann surface $M$, and 
with $\Ascr^1(M,\boldA)$ replaced by the bigger space
$\Oscr^1(M,\boldA)$ of holomorphic 1-forms on $M$ 
with values in $\boldA$, the homotopy principle for the maps in
\eqref{eq:diagram} was obtained by Forstneri\v c and L\'arusson
\cite[Theorems 1.1 and 1.2]{ForstnericLarusson2019CAG},
following the work of Alarc\'on and Forstneri\v c in
\cite{AlarconForstneric2014IM,AlarconForstneric2018Crelle}. 
In this case, the maps in \eqref{eq:diagram} were shown to be 
genuine homotopy equivalences when the homology group 
$H_1(M,\Z)$ is finitely generated, which holds 
if $M$ is an affine Riemann surface. 
We do not know whether the same is true in the 
context of Theorem \ref{th:whe1}. To show that a weak homotopy equivalence 
is a genuine homotopy equivalence using Whitehead's theorem, one needs to 
know that the spaces in question have the homotopy type of CW complexes.  
In \cite{ForstnericLarusson2019CAG}, this was established by showing that the spaces are absolute neighbourhood retracts (such results were first proved in \cite{Larusson2015PAMS}).  We do not see any way to prove this in the present setting and have no reason to believe it is true.  

The aforementioned results in \cite{ForstnericLarusson2019CAG}
are used in an important way in the proof of Theorem \ref{th:whe1}, 
given in Section \ref{sec:whe1}. 
Another major ingredient is the algebraic homotopy approximation
theorem for sections of algebraically elliptic submersions over an affine
manifold, due to Forstneri\v c \cite{Forstneric2006AJM}. We shall use
improved versions given by Theorems \ref{th:ahRunge1} 
and \ref{th:ahRunge2} and Corollary \ref{cor:flexible}. 

%
%
%
In order to understand the topological structure of the mapping spaces in \eqref{eq:diagram}, one may consider the following
extended diagram
\begin{equation}\label{eq:diagram2}
\xymatrix{
	 \CMI_*(M,\R^n) \ar[d]^\di \ar@{^{(}->}[r]^{\iota} &
	 \CMI_{\mathrm{nf}}(M,\R^n) \ar[d]^\di  & \\ 
	\Ascr^1(M,\boldA) \ar@{^{(}->}[r]^{\alpha} &
	\Oscr^1(M,\boldA) \ar[r]^{\beta} & \Cscr(M,\boldA),
}
\end{equation}
where $\iota$ is the inclusion of $\CMI_*(M,\R^n)$ in the space
$\CMI_{\mathrm{nf}}(M,\R^n)$ of nonflat conformal minimal
immersions $M\to\R^n$, $\alpha$ is the inclusion of the space
of algebraic 1-forms in the space of holomorphic 1-forms
with values in $\boldA$, and $\beta$ is the map 
$\phi\mapsto \phi/\theta$ where $\theta$ is a fixed 
nowhere vanishing holomorphic $1$-form on $M$. 
The map $\beta$ is a weak homotopy equivalence by the
Oka--Grauert principle since the null quadric $\boldA$ is 
complex homogeneous and hence an Oka manifold. 
The left-hand vertical map $\di$ is a weak homotopy equivalence by 
Theorem \ref{th:whe1}, while the right-hand one is a weak homotopy equivalence by \cite[Theorem 5.3]{ForstnericLarusson2019CAG}. 
In order to understand the inclusion $\iota$, it thus remains to understand the 
inclusion $\alpha$. In principle, the limitations of the algebraic Oka principle, 
discovered by L\'arusson and Truong \cite{LarussonTruong2019}, 
suggest that $\alpha$ may fail to be a weak homotopy equivalence.
Nevertheless, it has recently been shown by Alarc\'on and L\'arusson  
\cite[Proposition 2.3 and Corollary 1.8]{AlarconLarusson2023regular} 
that $\alpha$ induces a surjection of the path components of the two spaces, 
so $\iota$ does as well. We expect that $\alpha$ and hence $\iota$ induce 
bijections of path components, but we have not been able to prove it.
In Section \ref{sec:topological}, we reduce 
this question to the problem of whether the space 
$\Ascr^1(M,\boldA)$ is locally contractible; see
Theorem \ref{th:locally-contractible-2}.

For the inclusion $\Re \NC_*(M,\C^n) \hookrightarrow \CMI_*(M,\R^n)$,
Theorem \ref{th:whe1} says in particular that every complete nonflat conformal 
minimal immersion $M\to\R^n$ of finite total curvature can be deformed 
through maps of the same type to the real part of a proper algebraic null curve 
$M\to\C^n$. The following result shows that, in addition, one can control the 
flux along an isotopy in $\CMI_*(M,\R^n)$. 

%
%
\begin{theorem} \label{th:flux}
Let $M$ be an affine Riemann surface and $u_0:M\to\R^n$, $n\geq 3$, be a complete nonflat conformal minimal immersion of finite total curvature. Then there is a smooth isotopy $u_t:M\to\R^n$, $t\in[0,1]$, of complete nonflat conformal minimal immersions of finite total curvature such that $u_1$ is the real part of a proper algebraic null curve $h:M\to\C^n$. 
More generally, for any homotopy $\Fcal_t \in H^1(M,\R^n)$, $t\in [0,1]$, 
with $\Fcal_0=\Flux_{u_0}$ there is a smooth isotopy $u_t:M\to\R^n$, 
$t\in[0,1]$, of complete nonflat conformal minimal immersions of finite total 
curvature such that $\Flux_{u_t}=\Fcal_t$ for all $t\in [0,1]$.
\end{theorem}

This is an analogue of \cite[Theorem 1.1]{AlarconForstneric2018Crelle},
which gives a similar statement for nonflat minimal surfaces
without the finite total curvature condition; 
see also \cite[Theorem 1.1]{AlarconLarusson2022complete}, 
which permits prescribing the flux of all the surfaces in the isotopy
by using the same tools. If the immersion $u_0:M\to\R^n$ in 
Theorem \ref{th:flux} is not complete (which means that it extends
harmonically to an end of $M$), the theorem still holds with the weaker
conclusion that the nonflat conformal minimal immersions 
of finite total curvature $u_t:M\to\R^n$, $t\in[0,1]$, are complete 
for $c\le t\le 1$ for any given $c\in (0,1)$.
This shows that the inclusion of $\CMI_*(M,\R^n)$ into the 
space $\CMI_0(M,\R^n)$ of all nonflat conformal minimal immersions 
of finite total curvature $M\to\R^n$ (including the incomplete ones) 
induces a surjection of path components. 
We show that this inclusion is in fact a weak homotopy equivalence 
(see Corollary \ref{cor:whe-extended}).  We also prove a parametric version of Theorem \ref{th:flux} (see Theorem \ref{th:flux2}).

%
%
Holomorphic null curves are a special type 
of directed holomorphic immersions of open Riemann surfaces to 
complex Euclidean spaces. A connected compact complex submanifold 
$Y$ of $\P^{n-1}$, $n\geq 2$, determines the punctured complex cone \begin{equation}\label{eq:A}
	A= \{(z_1,\ldots,z_n)\in \C^n_* : [z_1:\cdots:z_n]\in Y\}.
\end{equation}
Note that $A$ is smooth and connected, and its closure 
$\overline A=A\cup\{0\} \subset \C^n$ is algebraic by Chow's theorem.
The map $\pi: A \to Y$ given by $\pi(z_1,\ldots,z_n)=[z_1:\cdots:z_n]$ is an algebraic $\C^*$-bundle, and by adding the zero section we obtain the restriction to $Y$ of the hyperplane bundle on $\P^{n-1}$.
A holomorphic immersion $h:M\to\C^n$ from an open Riemann surface 
is said to be {\em directed by} $A$, or an {\em $A$-immersion}, 
if its complex derivative with respect to any local holomorphic coordinate 
on $M$  takes its values in $A$. Equivalently, the differential $dh=\di h$ 
is a section of the subbundle $\Acal$ with fibre $A$ of the vector bundle 
$(T^*M)^{\oplus n}$. When $A$ is the null quadric $\boldA$, 
an $A$-immersion $M\to\C^n$ is the same thing as a
holomorphic null immersion.

Holomorphic directed immersions were studied by Alarc\'on and Forstneri\v c 
\cite{AlarconForstneric2014IM}. Under the assumption 
that $A$ is an Oka manifold not contained in any hyperplane in $\C^n$, 
they proved an Oka principle with Runge and Mergelyan approximation
for holomorphic $A$-immersions
\cite[Theorems 2.6 and 7.2]{AlarconForstneric2014IM}.  
(By \cite[Theorem 5.6.5]{Forstneric2017E}, the cone $A$ 
in \eqref{eq:A} is an Oka manifold if and only if $Y$ is.) 
They also showed that every holomorphic $A$-immersion can be
approximated by holomorphic $A$-embeddings when $n\ge 3$,
and by proper holomorphic $A$-embeddings under some natural extra assumptions on the cone $A$ \cite[Theorem 8.1]{AlarconForstneric2014IM}.
In the subsequent paper \cite{AlarconCastro-Infantes2019APDE} by Alarc\'on and Castro-Infantes, interpolation was added to the picture.  
A parametric Oka principle for holomorphic immersions 
directed by an Oka cone was proved as 
\cite[Theorem 5.3]{ForstnericLarusson2019CAG}. 

Recently, algebraic $A$-immersions from 
affine Riemann surfaces into $\C^n$ were studied in \cite{AlarconLarusson2023regular} under the assumption 
that the cone $A$ is algebraically elliptic in the sense of Gromov 
\cite{Gromov1989} (see also \cite[Definition 5.6.13]{Forstneric2017E});
we recall this notion in Section \ref{sec:approximation}. Several important 
cones arising in geometric applications, 
in particular the null quadric $\boldA$, are algebraically elliptic.
The optimal known sufficient condition for the cone
$A$ \eqref{eq:A} on a projective manifold $Y\subset \P^{n-1}$
to be algebraically elliptic is given by a recent theorem of 
Arzhantsev, Kaliman, and Zaidenberg 
\cite[Theorem 1.3]{ArzhantsevKalimanZaidenberg2024}.  
They showed that $A$ is algebraically elliptic if $Y$ 
(with $\dim Y\geq 1$) is uniformly rational, 
meaning that $Y$ is covered by Zariski-open sets, each isomorphic 
to a Zariski-open subset of affine space.

We obtain further results on directed algebraic immersions. 
The following Runge approximation 
and jet interpolation theorem is proved in Section \ref{sec:AAJI}.
It strengthens the main implication (iii) $\Rightarrow$ (ii)  
of \cite[Theorem 1.1]{AlarconLarusson2023regular}, upgrading  
plain interpolation to jet interpolation. This requires a new idea.

%
%
\begin{theorem}   \label{th:AAJI}
Let $A\subset\C_*^n$, $n\geq 2$, be a connected smooth punctured 
complex cone \eqref{eq:A} which is algebraically elliptic and not contained in a 
hyperplane in $\C^n$.  Let $M$ be an affine Riemann
surface, $\mathcal A$ be the subbundle of $(T^*M)^{\oplus n}$ defined by 
$A$, $K$ be a compact holomorphically convex subset of $M$, 
$U$ be an open neighbourhood of $K$, and 
$h:U\to\C^n$ be a holomorphic $A$-immersion such that 
the homotopy class of continuous sections of $\mathcal A|_U\to U$ 
that contains $dh$ also contains the restriction to $U$ of an algebraic section 
$\vartheta:M\to \mathcal A$. Then $h$ can be approximated 
uniformly on a neighbourhood of $K$ by proper algebraic 
$A$-immersions $\tilde h:M\to\C^n$ agreeing with $h$ to a given 
finite order on any given finite set in $K$ 
such that $d \tilde h$ is homotopic to $\vartheta$ through algebraic 
sections of $\mathcal A\to M$.
\end{theorem}

For the null quadric $\boldA$ \eqref{eq:nullq}, Theorem \ref{th:AAJI} was first 
proved by Alarc\'on and L\'opez  \cite[Theorem 1.2]{AlarconLopez2022APDE} 
using fairly technical results from the function theory of Riemann surfaces 
and the special geometry of the cone. 
In this case, the hypothesis that $dh$ be homotopic 
to the restriction of an algebraic section of $\Acal$ is always fulfilled
\cite[Proposition 2.3]{AlarconLarusson2023regular}. 
Their result also gives interpolation of poles at some ends of $M$
and includes Mergelyan approximation on admissible sets, as opposed 
to mere Runge approximation as in Theorem \ref{th:AAJI}. 
The latter is well understood and we do not repeat it here.

We recall the following notion from 
\cite[Definition 2.2]{AlarconForstneric2014IM}.
Let $A\subset\C_*^n$, $n\geq 2$, be a smooth punctured complex cone. 
We consider the tangent space $T_xA\subset T_x\C^n$ at a point $x\in A$ 
as a $\C$-linear subspace of $\C^n$.  Let $M$ be an open Riemann surface 
and $\theta$ be a nowhere vanishing holomorphic 1-form on $M$.
A holomorphic map $f:M\to A$ is {\em nondegenerate} if the linear span of the 
tangent spaces $T_{f(x)}A$, $x\in M$, equals $\C^n$.
A holomorphic $A$-immersion $h:M\to \C^n$ is nondegenerate  
if the holomorphic map $f=dh/\theta:M\to A$ is nondegenerate.
A holomorphic 1-form $\phi$ on $M$ with values in $A$ is 
nondegenerate if the holomorphic map $\phi/\theta:M\to A$ 
is nondegenerate.  Recall that a holomorphic null immersion is nondegenerate 
if and only if it is nonflat; see \cite[Lemma 2.5.3]{AlarconForstnericLopez2021} 
and the references therein. 

Given an open Riemann surface $M$, we denote by $\Iscr_{\rm nd}(M,A)$ 
the space of nondegenerate holomorphic $A$-immersions 
$M\to \C^n$, $n\geq 2$. Let $\theta$ be as above.
Assuming that the cone $A$ is an Oka manifold, it was proved in 
\cite[Theorem 5.6]{ForstnericLarusson2019CAG} 
(see also \cite[Theorem 3.12.7]{AlarconForstnericLopez2021}) 
that the map $\Iscr_{\rm nd}(M,A)\to \Oscr(M,A)$, $h\mapsto dh/\theta$, 
is a weak homotopy equivalence, and is a homotopy
equivalence if $M$ is of finite topological type, that is, if $H_1(M,\Z)$
is finitely generated. 

It is natural to ask to what extent this holds in the algebraic category, 
with $M$ an affine Riemann surface and $A\subset \C^n_*$, $n\ge 2$,
a connected smooth algebraic cone not contained in any hyperplane. 
If $A$ is algebraically elliptic, Theorem \ref{th:AAJI} shows that the map 
$\Iscr_*(M,A)\to \Ascr^1(M,A)$, $h\mapsto \di h$, 
from the space  $\Iscr_*(M,A)$ 
of proper nondegenerate algebraic $A$-immersions $M\to\C^n$ 
to the space $\Ascr^1(M,A)$ of algebraic 1-forms on $M$ with values in $A$ 
induces a surjection of path components. 
The next question is whether this map is injective on path components.
A partial answer is given by Theorem \ref{th:a-homotopy}.
Under the stronger assumptions that the cone $A$ is 
flexible, we have the following homotopy principle 
for nondegenerate algebraic $A$-immersions, 
analogous to Theorem \ref{th:whe1}. Recall (see Arzhantsev et al.\ 
\cite{ArzhantsevFlennerKalimanKutzschebauchZaidenberg2013DMJ}) 
that an algebraic manifold is said to be {\em flexible} if it carries finitely 
many complete algebraic vector fields with algebraic flows
spanning the tangent bundle at every point.

%
%
\begin{theorem}\label{th:whe2} 
Let $M$ be an affine Riemann surface, and let $A\subset \C^n_*$, 
$n\ge 2$, be a flexible smooth connected cone 
not contained in any hyperplane. Then the map  
\begin{equation}\label{eq:inclusion}
	\Iscr_*(M,A) \to \Ascr^1(M,A),\quad h\mapsto dh,
\end{equation}
is a weak homotopy equivalence.
\end{theorem}

The only way in which flexibility is used in this paper is that it implies
algebraic ellipticity with a trivial spray bundle; see 
Theorem \ref{th:ahRunge2} and the proof of Corollary \ref{cor:flexible}. 
The null quadric $\boldA$ is flexible
(see \cite[Proposition 1.15.3]{AlarconForstnericLopez2021},
which is a special case of a result of Kaliman
and Zaidenberg \cite[Section 5]{KalimanZaidenberg1999TG}),
so the part of Theorem \ref{th:whe1} pertaining to the
map $\Re \NC_*(M,\C^n)\to \Ascr^1(M,\boldA)$, $u\mapsto \di u$, 
is a special case of Theorem \ref{th:whe2}. The proof of Theorem \ref{th:whe2}
is obtained by exactly the same argument as in the 
case $A=\boldA$, so we omit the details.
The large cone $A=\C^n_*$, $n\ge 2$, is also flexible, so 
Theorem \ref{th:whe2} applies to it. 
Further examples are the cones $A$ defined by a flag manifold 
or a toric manifold $Y$ (see Arzhantsev et al.\  
\cite{ArzhantsevZaidenbergKuymzhiyan2012}). 
Other examples are given in \cite{Perepechko2013,MichalekPerepechkoSuss2018,ProkhorovZaidenberg2023}.

%
%
The paper is organised as follows.
In Section \ref{sec:approximation} we prove the algebraic
homotopy approximation theorems used in the paper. 
In Section \ref{sec:AAJI} we prove Theorem \ref{th:AAJI},
and in Section \ref{sec:flux} we prove Theorem \ref{th:flux}.
These results concern the nonparametric situation, and the main ideas 
are explained more easily in this case.
Theorem \ref{th:whe1} is proved in Section \ref{sec:whe1}
as a consequence of the h-principle in Theorem \ref{th:hp1}; 
essentially the same proof yields Theorem \ref{th:whe2}. 
In Section \ref{sec:flux2}, we strengthen the results of 
Section \ref{sec:whe1} by including control of flux;  
see Theorem \ref{th:flux2}. We show 
in particular that the flux homomorphism 
$\Flux: \CMI_*(M,\R^n)\to H^1(M,\R^n)$ is a Serre fibration for every 
affine Riemann surface $M$; see Corollary \ref{cor:flux2}. 
In Section \ref{sec:A-homotopies},   
we prove a special case of Theorem \ref{th:whe2} assuming only 
that the cone $A$ is algebraically elliptic.
To this end, we introduce the notion of an algebraic homotopy.
In Section \ref{sec:topological}, we show that stronger 
conclusions regarding the homotopy type of the spaces
under consideration would follow from local contractibility of the spaces 
$\Ascr^1(M,A)$ with respect to the compact-open topology, where $M$
is an affine Riemann surface and $A\subset \C^n_*$
is a flexible punctured cone. Whether or not the space 
$\Ascr^1(M,A)$ is locally contractible in general remains 
an important open problem. 

%
%
%
%
\section{Algebraic homotopy approximation theorems}
\label{sec:approximation}

\noindent
In this section we prove results on approximating families of holomorphic
sections of an algebraically subelliptic or flexible submersion over an affine
variety by families of algebraic sections; 
see Theorems \ref{th:ahRunge1}, \ref{th:ahRunge2} and 
Corollary \ref{cor:flexible}.
These results are applications and extensions 
of a theorem of Forstneri\v c \cite[Theorem 3.3]{Forstneric2006AJM} 
(see also \cite[Theorem 6.15.3]{Forstneric2017E}). 
They provide some of the main tools used in the paper.

Let $\hgot: Z\to X$ be an algebraic submersion from an algebraic
variety $Z$ onto an affine algebraic variety $X$.
We recall the notions of algebraic ellipticity and subellipticity
introduced by Gromov \cite[1.1.B]{Gromov1989}; 
see also \cite[Definition 6.1.1]{Forstneric2017E}.

An algebraic fibre spray on $Z$ is a triple $(E,\pi,s)$ where
$\pi:E\to Z$ is an algebraic vector bundle and $s:E\to Z$
is an algebraic map (the spray map) such that 
\[
    s(0_z)=z\ \ {\rm and}\ \  s(E_z) \subset Z_{\hgot(z)}= \hgot^{-1}(\hgot(z))
    \ \ \text{for every $z\in Z$}.
\]
Here, $E_z=\pi^{-1}(z)$ is the fibre of $E$ over $z$ and $0_z\in E_z$
denotes the origin. A fibre spray $(E,\pi,s)$ is said to be 
{\em dominating at a point} $z\in Z$ if the differential 
$d s_{0_z} : T_{0_z} E \to T_z Z$ maps the subspace 
$E_z\subset T _{0_z} E$ surjectively onto the vertical tangent space 
$\ker d \hgot_z$. The fibre spray is said to be
{\em dominating} if it is dominating at every point $z\in Z$.
A family of fibre sprays $(E_j,\pi_j,s_j)$ $(j=1,\ldots,m)$ on $Z$ 
is said to be dominating at the point $z\in Z$ if
\begin{equation}\label{eq:fibre-domination}
    (ds_1)_{0_z}(E_{1,z}) + (ds_2)_{0_z}(E_{2,z})+ \cdots
                     + (d s_m)_{0_z}(E_{m,z})= \ker(d\hgot_z),
\end{equation}
and to be {\em dominating} if this holds at every point $z\in Z$. 
The algebraic submersion $\hgot:Z\to X$
is {\em algebraically elliptic} if it admits a dominating algebraic fibre spray,
and is {\em algebraically subelliptic} if it admits a dominating family 
of such sprays. An algebraic manifold $Y$ is algebraically
(sub-)elliptic if the submersion from $Y$ to a point is.
Recently, Kaliman and Zaidenberg  \cite{KalimanZaidenberg2024FM}
showed that algebraic ellipticity and algebraic subellipticity are equivalent. 
However, the proofs of the results in this section do not
depend on this fact, and it is usually easier to find dominating families 
of sprays than a dominating spray.

We begin with the following approximation result for sections
of algebraically subelliptic submersions.

%
%
\begin{theorem}
\label{th:ahRunge1}
Let $\hgot: Z\to X$ be an algebraically subelliptic submersion from an algebraic
manifold $Z$ onto an affine algebraic manifold $X$. Fix a distance
function $\dist$ on $Z$ inducing its standard Hausdorff topology. 
Let $K\subset X$ be a compact $\Oscr(X)$-convex subset, 
$B\subset \C^N$ be a compact convex set containing the origin, 
$U\supset K$ and $V\supset B$ be open neighbourhoods,
and $f_{\zeta,t} : U \to Z|_U$ $(\zeta\in V,\ t\in [0,1]^n,\ n\geq 1)$ 
be a continuous family of holomorphic sections of $\hgot:Z|_U \to U$   
depending holomorphically on $\zeta\in V$ 
such that $f_{0,0}$ is homotopic through sections of $Z|_U\to U$
to the restriction of an algebraic section $X\to Z$. 
Given $\epsilon >0$, there is an algebraic map 
$F:X\times \C^N \times \C^n \to Z$ satisfying the following conditions:
\begin{enumerate}[\rm (a)] 
\item $\hgot(F(x,\zeta,t))=x$ for all $x\in X$, $\zeta\in\C^N$, and $t\in \C^n$.
\item $\dist(F(x,\zeta,t),f_{\zeta,t}(x))<\epsilon$ for every 
$x\in K$, $\zeta\in B$, and $t\in [0,1]^n$.
\item If $f_{0,0}=f|_U$ for an algebraic section $f:X\to Z$, then $F$
can be chosen such that $F(\cdotp, 0,0)=f$. 
\item Assume that {\rm (c)} holds. 
If $X_0$ is a closed algebraic subvariety of $X$ and 
$f_{0,t}$ is independent of $t$ on $X_0\cap U$, 
then $F$ can be chosen such that in addition $F(\cdotp,0,t)$ 
agrees with $f$ on $K\cap X_0$ for all $t\in \C^n$.
\end{enumerate}
\end{theorem}

If we wish to ignore $B$, that is, holomorphic dependence on a parameter in $B$ is not required, then we take $N=0$ and $\C^N=B=V$ to be a point.

\begin{proof} 
The basic case when $n=1$, $N=0$ (so $\C^N=B=V$ is a point), 
$f_0$ extends to an algebraic section $f:X\to Z$, and without condition 
(d) is given by \cite[Theorem 3.1]{Forstneric2006AJM}
(see also \cite[Theorem 6.15.3]{Forstneric2017E}).  
An inspection of the proof also gives part (d); 
cf.\ \cite[Theorem 6.4]{Forstneric2023Indag}. 
See also Remark \ref{rem:Mergelyan} below.
The multiparameter case $t\in [0,1]^n$ with $n>1$ and $f_0=f|_U$ for an
algebraic section $f:X\to Z$ is obtained by induction on the number $n$ 
of parameters $t=(t_1,\ldots,t_n)$; see 
\cite[Theorem 4.2]{ForstnericLarusson2024Oka1}.

If $f_0$ is merely homotopic to the restriction $f|_U$ of an algebraic
section $f:X\to Z$ (the assumption in parts (a) and (b) of
the theorem), we proceed as follows, still assuming that 
$B$ is a point.  By shrinking $U$ around $K$ 
we may assume that $U$ is a Stein domain. 
By the Oka principle, the homotopy from $f_0$ to $f|_U$ 
can be chosen to consist of holomorphic sections 
\cite[Theorem 5.4.4]{Forstneric2017E}.
We parametrise this homotopy by the segment 
$T=\{(s,\ldots,s)\in \R^n: -1\le s\le 0\}$, so that $f|_U$ corresponds to the 
point $(-1,\ldots,-1)\in T$ and $f_0$ to the origin $0=(0,\ldots,0)$. 
Let $C\subset \R^n$ denote the convex hull
of $T\cup [0,1]^n$. There is an affine diffeomorphism
$\psi:\R^n\to\R^n$ such that $\psi([0,1]^n)=C$, 
$\psi(0,\ldots,0)=(-1,\ldots,-1)$, and $\psi$ fixes the 
remaining vertices of $[0,1]^n$. Note that $\psi$ complexifies
to a complex affine isomorphism $\psi:\C^n\to \C^n$.
Clearly, there is a strong deformation retraction 
$\tau:C \to T\cup [0,1]^n$. In particular, $\tau$ restricts to the identity 
map on $T\cup [0,1]^n$. Let $L=\psi^{-1}([0,1]^n) \subset [0,1]^n$. 
The composition $\phi=\tau\circ \psi: [0,1]^n\to C$ is a surjective map with 
$\tau\circ \psi(0,\ldots,0)=(-1,\ldots,-1)$, which maps $L$ onto
$[0,1]^n$ by $\psi$ (since $\tau$ is the identity on $[0,1]^n$). 
We now apply the already proved case of the theorem 
to the homotopy $\tilde f_t =f_{\phi(t)}:U\to Z|_U$ with $t\in [0,1]^n$
to approximate it by the restriction of an algebraic map
$\wt F:M\times \C^n\to Z$. Then the algebraic map $F:M\times \C^n\to Z$
given by $F(x,t) = \wt F(x,\psi^{-1}(t))$ satisfies conditions (a) and (b) 
in the theorem, while conditions (c) and (d) are vacuous.

Assume now that $N>0$. We may take $V\supset B$ to be an open convex 
domain in $\C^N$. Note that a family of holomorphic sections 
$f_{\zeta,t}:U\to Z|_U$ $(t\in [0,1]^n)$ of $\hgot:Z|_U\to U$, 
depending holomorphically on $\zeta\in V$ and continuously on $t\in [0,1]^n$, 
is the same thing as a homotopy of holomorphic sections 
$\tilde f_t$ $(t\in [0,1]^n)$ of the algebraic submersion 
\begin{equation}\label{eq:tilde-h}
	\tilde \hgot: \wt Z=Z\times \C^N\to \wt X = X\times \C^N 
\end{equation}
over the domain $U\times V\subset \wt X$, where 
\[
	\tilde \hgot(z,\zeta)=(\hgot(z),\zeta)\ \ \text{and}\ \ 
	\tilde f_t(x,\zeta) =(f_{\zeta,t}(x),\zeta)
	\ \ \text{for $x\in U$ and $\zeta\in V$}.
\]  
Note that the set $K\times B$ is holomorphically convex in $M\times \C^N$.
By trivially extending the algebraic fibre sprays $(E_j,\pi_j,s_j)$ in the 
hypothesis of the theorem (see \eqref{eq:fibre-domination}) to algebraic 
fibre sprays on $\wt Z$ we see that the extended submersion 
$\tilde \hgot:\wt Z\to \wt X$ is also algebraically subelliptic. 
We introduce an extra parameter $s\in [0,1]$ and consider the
continuous family of holomorphic sections 
$f'_{s,t}:U\times V\to \wt Z$ of $\tilde \hgot$ defined by 
\[
	f'_{s,t}(x,\zeta) = (f_{s\zeta,t}(x),\zeta) \in \wt Z\ \  
	\text{for}\ \ x\in U,\ \zeta\in V,\ s\in\R,\ t\in [0,1]^{n}.
\]
Note that $f'_{1,t}=\tilde f_{t}$, $f'_{0,t}(x,\zeta)=(f_{0,t}(x),\zeta)$, and 
$f'_{0,0}(x,\zeta)=(f_{0,0}(x),\zeta)$ extends to the algebraic section 
$(x,\zeta)\mapsto (f(x),\zeta)$ of the submersion \eqref{eq:tilde-h}.
The already established case of the theorem with 
approximation on $K\times B\times [0,1]^{n+1}$ yields an algebraic
map $\wt F :X\times \C^N \times  \C^{n+1} \to \wt Z$ satisfying
the conclusion of the theorem with respect to the homotopy $f'_{s,t}$ 
(with $(s,t)\in [0,1]^{n+1}$)  of holomorphic sections of the submersion 
\eqref{eq:tilde-h}. Setting $s=1$ and postcomposing $\wt F$ by the trivial
projection $\wt Z=Z\times \C^N\to Z$ yields an algebraic map
$F : X\times \C^N  \times \C^{n} \to Z$ satisfying the 
theorem with respect to the submersion $\hgot:Z\to X$.
\end{proof}

\begin{proposition}
Every algebraic fibre bundle $\hgot : Z\to X$  with 
algebraically elliptic fibre is algebraically subelliptic.
Hence, Theorem \ref{th:ahRunge1} applies in this case.
\end{proposition}

\begin{proof}
Note that $X$ is covered by finitely many Zariski open domains $U_i$
such that the restriction $\hgot:Z|_{U_i}\to U_i$ is a trivial fibre bundle 
with algebraically elliptic fibre $A$ for every $i$, so it admits a 
dominating algebraic fibre spray obtained by a trivial extension 
of a dominating algebraic spray on $A$.
A straightforward generalisation of \cite[Proposition 6.4.2]{Forstneric2017E} 
gives the localisation of algebraic subellipticity from sprays to fibre sprays,
and hence we obtain finitely many algebraic fibre sprays on $Z$ which satisfy
the hypotheses of the theorem. The argument is the same
as in the proof of \cite[Proposition 6.4.2]{Forstneric2017E}. 
The only difference is that, in the case of fibre sprays, the spray maps 
into the total space of the fibre bundle must respect the fibres. 
\end{proof}

Since the null quadric $\boldA$ \eqref{eq:nullq} is flexible
and hence algebraically elliptic,
we have the following immediate corollary.
The last statement is \cite[Proposition 3.1]{AlarconLarusson2023regular};
it can also be seen as an immediate corollary of the results in \cite{Lopez2014TAMS,AlarconLopez2022APDE}.

\begin{corollary}\label{cor:h-Runge-algebraic}
Let $M$ be an affine Riemann surface and $\Acal \to M$ be 
the algebraic subbundle of $(T^*M)^{\oplus n}$ whose fibre is the null quadric 
$\boldA\subset\C_*^n$, $n\geq 3$. Then the conclusion of 
Theorem \ref{th:ahRunge1} holds for sections $M\to \Acal$. 
Furthermore, every section of $\Acal$ over $M$ is homotopic to an 
algebraic section.
\end{corollary}

%
%
\begin{remark}\label{rem:Mergelyan}
As noted already, the basic case of Theorem \ref{th:ahRunge1}
with the parameter $t\in[0,1]$ and without $\zeta$-parameters coincides 
with \cite[Theorem 3.3]{Forstneric2006AJM} 
(see also \cite[Theorem 6.15.3]{Forstneric2017E}). 
In the last part of the proof of this result in \cite[p.\ 251]{Forstneric2006AJM} 
the author referred to a small extension of the Oka--Weil theorem 
without being explicit. The point is to approximate the given homotopy
$\{\eta^m_t\}_{t\in[0,1]}$ (using the notation in the cited paper)
of holomorphic sections of a trivial vector bundle
over $X$, depending continuously on $t\in[0,1]$, by a homotopy which is 
holomorphic in $t$ in a neighbourhood of $[0,1]\subset \C$, and 
applying the Oka--Weil theorem to approximate by a polynomial
map in all variables. We spell out this argument in greater generality 
since we shall need it. In the next result, the relevant parameter
is called $p$.
\end{remark}

%
%
\begin{proposition}\label{prop:Mergelyan}
Let $Z$ and $X$ be complex manifolds, 
$\hgot: Z\to X$ be a holomorphic submersion onto $X$, $K$ be a compact 
subset of $X$ with a basis of Stein neighbourhoods, and 
$U\subset X$ be an open set containing $K$. 
Assume that  $f: U\times [0,1] \to Z$ is a continuous map such that 
$\hgot(f(x,p))=x$ for all $(x,p)\in U\times [0,1]$ 
and the section $f_p=f(\cdotp,p):U\to Z|_U$ of $\hgot$ 
is holomorphic for every $p\in [0,1]$. Then 
we can approximate $f$ uniformly on $K\times [0,1]$ by holomorphic maps  
$F:U' \times V' \to Z$, defined on open neighbourhoods
$U'\times V'\subset X\times\C$ of $K\times [0,1]$, such that 
$\hgot(F(x,p))=x$ for all $(x,p)\in U'\times V'$ and 
$F(\cdotp,p)=f(\cdotp,p)$ for $p=0,1$.
Furthermore, a homotopy $f^t:U\times [0,1] \to Z$ of maps
as above, depending continuously on $t\in [0,1]$, can be
approximated by a homotopy of holomorphic maps $F^t:U' \times V' \to Z$ having the stated properties for every $t\in [0,1]$.
\end{proposition}

\begin{proof}
Let us first consider the special case when $\hgot:Z=X\times \C\to X$
is the trivial submersion with fibre $\C$. In this case, a section 
of $\hgot$ can be identified with a map to $\C$.
Thus, let $f:U\times [0,1]\to \C$ be a continuous function such that 
$f(\cdotp,p)$ is holomorphic on $U$ for every $p\in [0,1]$. 
Choose a partition $0=p_0<p_1<p_2<\cdots < p_k=1$ of $[0,1]$ 
and set $p_{-1}=-1$ and $p_{k+1}=2$. Let $1=\sum_{i=0}^k \chi_i$
be a partition of unity on $[0,1]$, where $\chi_i:\R \to [0,1]$ is a
smooth function supported in $(p_{i-1},p_{i+1})$ for each $i=0,1,\ldots,k$.
Note that $\chi_i(p_i)=1$ for $i=0,1,\ldots,k$. By a theorem of
Weierstrass we can approximate $\chi_i$ as closely as desired
uniformly on $[-1,2]$ by a holomorphic polynomial $g_i$ on $\C$
with $g_i(p_i)=1$. Assuming that the partition $\{p_i\}$ is fine enough
and the approximation of $\chi_i$ by $g_i$ is close enough, the 
holomorphic function
\[
	F(x,p)=\sum_{i=0}^k f(x,p_i) g_i(p), \quad x\in U,\ p\in \C,
\]
approximates $f$ as closely as desired on $U\times [0,1]$.

We now consider the general case. 
The assumption on the compact set $K\subset X$ implies that it 
has a basis of open Stein neighbourhoods in $X$. Hence, we may 
assume that the domain $U\subset X$ is Stein and $K$ is $\Oscr(U)$-convex.
Consider the graph of $f$ over $K\times [0,1]$:
\[
	\Sigma =\bigl\{(f(x,p),p) : x\in K,\ p\in [0,1]\bigr\} \subset Z \times \C.
\]
By \cite[Corollary 2.2]{ForstnericWold2010PAMS}
(see also \cite[Corollary 3.6.6]{Forstneric2017E}), 
$\Sigma$ admits an open Stein neighbourhood $\Theta$ in $Z\times\C$
such that $\Sigma$ is $\Oscr(\Theta)$-convex.
By the standard method, embedding $\Theta$ in a Euclidean 
space and using a holomorphic retraction from a neighbourhood of 
$\Sigma\hra \C^N$ onto $\Sigma$, we reduce the proof 
to the already established case of approximating functions.
(The details of this argument can be found for example in 
\cite[proof of Theorem 3.8.1, p.\ 90]{Forstneric2017E}.)
The same proof applies to homotopies $f^t$ $(t\in [0,1])$ by 
using a partition of unity in the parameter.
\end{proof}

We now modify Theorem \ref{th:ahRunge1} under the additional
assumption that the algebraic submersion $\hgot:Z\to X$ admits
a fibre dominating algebraic spray $(E,\pi,s)$ defined on a trivial vector bundle
$\pi:E=Z\times \C^r \to Z$.  Under this assumption, we have the following 
approximation result in which the data and the approximants depend continuously on a parameter in a compact Hausdorff space.

%
%
\begin{theorem}
\label{th:ahRunge2}
Let $\hgot: Z\to X$ be an algebraic submersion from an algebraic
variety $Z$ onto an affine algebraic variety $X$. 
Assume that there is a fibre dominating algebraic spray $(E,\pi,s)$ 
on $Z$ defined on a trivial vector bundle $\pi:E=Z\times \C^r \to Z$. 
Fix a distance function $\dist$ on $Z$ inducing its standard Hausdorff topology. 
Let $K\subset X$ be a compact $\Oscr(X)$-convex 
subset, $B\subset \C^N$ be a compact convex set 
containing the origin, $U\supset K$ and $V\supset B$ 
be open neighbourhoods, $Q$ be a closed subspace of a compact Hausdorff space $P$, and $f_{\zeta,p,t} : U \to Z|_U$ $(\zeta\in V,\ p\in P,\ t\in [0,1])$ 
be a continuous family of holomorphic sections of $\hgot:Z|_U \to U$   
depending holomorphically on $\zeta\in V$ such that $f_{0,p,0}:X\to Z$ 
is an algebraic section for every $p\in P$ and 
$f_{0,p,t}=f_{0,p,0}$ holds for every $(p,t) \in Q\times [0,1]$.

Given $\epsilon >0$, there is a continuous map 
$F:X\times \C^N \times P\times \C \to Z$ such that the map
$F(\cdotp,\cdotp,p,t):X\times \C^N  \to Z$ 
is algebraic for every $(p,t)\in P\times [0,1]$
and the following conditions hold:
\begin{enumerate}[\rm (a)] 
\item $\hgot(F(x,\zeta,p,t))=x$ for all $x\in X$, $\zeta\in\C^N$, $p\in P$,
and $t\in \C$. 
\item $\dist(F(x,\zeta,p,t),f_{\zeta,p,t}(x))<\epsilon$ for every 
$x\in K$, $\zeta\in B$, $p\in P$, and $t\in [0,1]$.
\item $F(\cdotp,0,p,t)=f_{0,p,t}$ for every 
$(p,t)\in (P\times \{0\}) \cup (Q\times [0,1])$. 
\end{enumerate}
\end{theorem}

\begin{proof}
Since the vector bundle $\pi:E\to Z$ is assumed to be trivial, 
any iterated spray bundle $(E^{(k)},\pi^{(k)},s^{(k)})$ $(k\in\N)$ 
obtained from $(E,\pi,s)$ (see \cite[Definition 6.3.5]{Forstneric2017E}) is 
also a trivial bundle over $Z$. Recall that $f_{0,p,0} : X \to Z$ is an algebraic
section for every $p\in P$, and hence its graph 
\[
	V_p = \{f_{0,p,0}(x): x\in X\} \subset Z
\]
is an affine algebraic submanifold of $Z$ isomorphic to $X$.
After shrinking the neighbourhoods $U\supset K$ and $V\supset B$ 
if necessary, \cite[Proposition 6.5.1]{Forstneric2017E} gives 
an integer $k\in\N$ and a homotopy of holomorphic maps
$\eta_{\zeta,p,t}:U \to\C^{kr}$ depending continuously
on $(p,t)\in P\times [0,1]$ and holomorphically on $\zeta\in V$, 
where $kr$ is the rank of the iterated  (trivial) vector bundle 
$E^{(k)}\cong Z\times\C^{kr}\to Z$, such that 
$\eta_{0,p,t}$ is the zero map on $U$ for every 
$(p,t)\in (P\times \{0\})\cup (Q\times [0,1])$ and 
\[
	s^{(k)}(f_{0,p,0}(x), \eta_{\zeta,p,t}(x)) = f_{\zeta,p,t}(x)\ \ 
	\text{for all}\ x\in U,\ \zeta \in V,\ p\in P,\ \text{and}\ t\in [0,1].
\]
Since $X$ is an affine algebraic manifold, we can apply the 
parametric Oka--Weil theorem \cite[Theorem 2.8.4]{Forstneric2017E} 
to approximate the homotopy $\eta_{\cdotp,p,t}$, uniformly
on $K\times B$ and uniformly in the parameters $(p,t)\in P\times [0,1]$, 
by a homotopy of algebraic maps 
$\wt \eta_{\cdotp,p,t}:X\times \C^N\to \C^{kr}$
depending continuously on $(p,t)\in P \times [0,1]$, 
with $\wt \eta_{0,p,t}$ the zero map on $X$ for every 
$(p,t)\in (P\times \{0\})\cup (Q\times [0,1])$. 
Since the spray map $s^{(k)}:E^{(k)}\to Z$ is algebraic, the map
$F:X\times \C^N \times P\times \C \to Z$ defined by 
\[
	F(x,\zeta,p,t) = s^{(k)}(f_{0,p,t}(x), \wt\eta_{\zeta,p,t}(x))
\]
satisfies the conclusion of the theorem provided that 
the approximation of $\eta_{\cdotp,p,t}$ by 
$\wt \eta_{\cdotp,p,t}$ was close enough. 
\end{proof}

\begin{remark}
As in Theorem \ref{th:ahRunge1}, 
we can ensure that the map $F$ in Theorem \ref{th:ahRunge2} 
is also algebraic in $t\in \C$; see Remark \ref{rem:Mergelyan} and 
Proposition \ref{prop:Mergelyan}. 
\end{remark}

The following corollary to Theorem \ref{th:ahRunge2} 
is of particular importance in this paper.
Recall that a complete algebraic vector field with algebraic
flow on an algebraic manifold is called a {\em locally nilpotent
derivation}, abbreviated LND. 

\begin{corollary}\label{cor:flexible}
The assumptions, and hence the conclusion of Theorem \ref{th:ahRunge2}, 
hold if the algebraic submersion $\hgot:Z\to X$ is flexible, in the sense 
that $Z$ admits finitely many LNDs tangent to the fibres of 
$\hgot$ and spanning the tangent space of the fibre 
at every point of $Z$. In particular, Theorem \ref{th:ahRunge2} 
holds for sections of any algebraic fibre bundle $Z\to X$ with 
flexible fibre $Y$ over an affine algebraic manifold $X$. 
\end{corollary}

\begin{proof}
If $V_1,\ldots, V_r$ are LNDs on $Z$ which satisfy the hypothesis
of the corollary and $\phi_t^j$ denotes the (algebraic) flow of $V_j$
for time $t\in \C$, then the algebraic map $s:Z\times \C^r\to Z$
given by 
$
	s(z,t_1,\ldots,t_r)=\phi_{t_1}^1\circ \phi_{t_2}^2\circ
	\cdots\circ \phi_{t_r}^r(z)
$ 
($z\in Z,\ t_1,\ldots,t_r\in \C$) is an algebraic fibre dominating spray.

Assume now that $\hgot:Z\to X$ is an algebraic fibre bundle with a 
flexible fibre $Y$ over an affine manifold $X$. If $U\subset X$ 
is a Zariski open subset such that the restricted bundle $\hgot:Z|_U\to U$ 
is algebraically trivial, then clearly this restricted bundle is flexible.
Let $V_1,\ldots,V_r$ be LNDs on $Z|_U\cong U\times Y$ 
which are tangent to the fibres of $\hgot$ and span the fibre 
at every point of $Z|_U$.
If $g:X\to\C$ is an algebraic function that vanishes to a sufficiently 
high order on the subvariety $X\setminus U$ then the products
$g V_i$ for $i=1,\ldots,r$ extend to LNDs on $Z$ that vanish 
on $\hgot^{-1}(\{g=0\})$. Finitely many such collections of LNDs 
on $Z$, for different vector bundle charts $U$ and functions $g$, 
show that $\hgot:Z\to X$ is flexible. 
\end{proof}

%
%
\section{Approximation and interpolation  
of directed algebraic immersions}\label{sec:AAJI}

\begin{proof}[Proof of Theorem \ref{th:AAJI}]
Recall that $A\subset\C_*^n$, $n\geq 2$, is assumed to be a connected 
smooth punctured cone \eqref{eq:A} which is algebraically elliptic and 
not contained in any hyperplane in $\C^n$.  Let $M$ be an affine Riemann
surface, $\mathcal A$ be the subbundle of $(T^*M)^{\oplus n}$ defined by 
$A$, $K$ be a holomorphically convex compact subset of $M$, 
$U\supset K$ be an open neighbourhood of $K$, and 
$h:U\to\C^n$ be a holomorphic $A$-immersion such that 
the homotopy class of continuous sections of $\mathcal A|_U$ 
that contains $dh$ also contains the restriction of an algebraic section 
$\vartheta$ of $\mathcal A$ on $M$. 
By \cite[Theorem 2.3 (a)]{AlarconForstneric2014IM}
and \cite[Theorem 1.3]{AlarconCastro-Infantes2019APDE},
we may assume that $h$ is nondegenerate. 
Let $Q=\{q_1,\ldots,q_l\}\subset K$ be the 
points at which we wish to interpolate. 

By \cite[Theorem 2.6]{AlarconForstneric2014IM} we can approximate
$h$ uniformly on $K$ by a holomorphic $A$-immersion $\tilde h:M \to\C^n$
whose differential $d\tilde h$ is homotopic to $\vartheta$.
Furthermore, by Alarc\'on and Castro-Infantes 
\cite{AlarconCastro-Infantes2019APDE} we may assume that 
$\tilde h$ agrees with $h$ to the given finite order $k$ in the points of $Q$. 
We replace $h$ by $\tilde h$ and drop the tilde.

Choose a compact set 
$L\subset M=\overline M\setminus\{x_1,\ldots,x_m\}$ 
containing $K$ in its interior such 
that $M\setminus L = \bigcup_{j=1}^m D_j\setminus \{x_j\}$, where 
$D_1,\ldots,D_m$ are discs in $\overline M$ with 
pairwise disjoint closures such that $x_j\in D_j$ for $j=1,\ldots,m$. 
Note that the inclusion $L\hra M$ is a homotopy equivalence; in particular,
$L$ is connected. Choose a collection $\Cscr=\{C_1,C_2,\ldots,C_s\}$ 
of smooth embedded arcs and closed curves in $\mathring L=L\setminus bL$
which contains a basis of the homology group $H_1(L,\Z)=H_1(M,\Z)$,  
and also arcs connecting $q_1$ to each of the points $q_2,\ldots,q_l\in Q$,  
such that any two curves in $\Cscr$ only meet in $q_1$ and 
$|\Cscr|=\bigcup_{i=1}^s C_i \subset \mathring L$ is a connected 
compact Runge set in $M$. For the construction of such a family,
see \cite[Lemma 1.12.10]{AlarconForstnericLopez2021} 
and \cite[proof of Proposition 3.3.2, p.\ 142]{AlarconForstnericLopez2021}.

Choose a nowhere vanishing holomorphic $1$-form $\theta$
on $M$ and write $dh=f\theta$, so $f:M \to A$ is a nondegenerate 
holomorphic map. We will construct a holomorphic deformation family of $f$ 
on a neighbourhood $V\subset M$ of $L$ (that is, a spray of holomorphic 
maps $V\to A$) which is both period dominating with respect to the family 
$\Cscr$ and $k$-jet dominating on the set 
$Q=\{q_1,\ldots,q_l\}$, where $k$ is the desired order of interpolation. 
 
We begin by explaining the second issue. Given points $q\in M$ 
and $z\in A$, we denote by $J_{q,z}^k$ the space of 
$k$-jets of holomorphic maps from open neigbourhoods of $q\in M$ 
to $A$ sending $q$ to $z$. In local holomorphic coordinates, this is the space
of holomorphic polynomial maps of degree $k$ in one complex variable 
with $\dim A$ components and with vanishing constant term, so it is
a complex Euclidean space of some dimension $s=s(k,\dim A)$.
By varying the point $z\in A$ we get the complex manifold
$J_{q,A}^k=\bigsqcup_{z\in A} J_{q,z}^k$ of dimension $s+\dim A$.  
Given a holomorphic map $f:M \to A$, we 
denote by $j^k_q(f)\in J_{q,A}^k$ the $k$-jet of $f$ at $q$
(including its value $f(q)\in A$). Finally, set 
\[
	j_Q^k(f) = \bigl( j_{q_1}^k(f),\ldots,j_{q_l}^k(f)\bigr) 
			\in \prod_{i=1}^l J_{q_i,A}^k=:J_{Q,A}^k.
\]
Note that $J_{Q,A}^k$ is a complex manifold of dimension $N_1=l(s+\dim A)$.

By a theorem of Siu \cite{Siu1976} 
(see also \cite[Theorem 3.1.1]{Forstneric2017E}), 
the graph of a holomorphic map $f:M \to A$ 
over $L$ has an open Stein neighbourhood in $M\times A$, 
and the normal bundle to the graph is trivial by the Oka--Grauert principle
\cite[Theorem 5.3.1 (iii)]{Forstneric2017E}. Hence, the space
of holomorphic maps $L\to A$ near a given map $f:L\to A$ is 
isomorphic to a neighbourhood of the origin in the 
space of holomorphic maps $L\to \C^{\dim A}$.
By using standard tools we can find a neighbourhood $V\supset L$
and an $N_1$-parameter holomorphic spray of maps 
$f_\zeta: V \to A$ for $\zeta$ in a ball $0\in B_1\subset \C^{N_1}$
such that $f_0=f|_V$ and the differential
\[
		\frac{\di}{\di\zeta}\bigg|_{\zeta=0} j^k_{Q}(f_\zeta):
		\C^{N_1} \longrightarrow T_{j^k_{Q}(f)} J_{Q,A}^k \cong \C^{N_1}
\]
is an isomorphism. 
This means in particular that for every collection of values and $k$-jets near the values and $k$-jets of $f$ at the points $q_i\in Q$ $(i=1,\dots,l)$ there is precisely one member $f_\zeta$ of this family which assumes these values and $k$-jets at the points of $Q$. 
Thus, $\{f_\zeta\}_{\zeta\in B_1}$ is a universal local deformation 
family of $f=f_0$ with respect to the values and $k$-jets at the points of $Q$.

Let $\Cscr=\{C_1,C_2,\ldots,C_s\}$ be the family of curves in $L$
constructed above. On each curve $C_i$ we choose an orientation.
For every continuous map $g:|\Cscr|\to A$ we define its periods by
\begin{equation}\label{eq:periodmap}
	\Pcal(g) = \left(\int_{C_i} g \theta\right)_{i=1,\ldots,s} \in 
	(\C^n)^s = \C^{ns}.
\end{equation}
Set $N_2=ns$.
Recall that the given immersion $h:M\to \C^n$ is nondegenerate. 
Hence, after shrinking the ball $B_1\subset \C^{N_1}$ around 
the origin and the open set $V\supset L$ around $L$ if necessary, 
we can apply \cite[Lemma 3.2.1]{AlarconForstnericLopez2021}
(whose basic version is \cite[Lemma 5.1]{AlarconForstneric2014IM}) 
to extend the spray $f_\zeta$ to another holomorphic spray of maps 
$f_{\zeta,\tau}:V\to A$, where $\tau$ belongs to a ball $B_2\subset \C^{N_2}$ 
around the origin, such that 
\begin{eqnarray*}
	f_{\zeta,0} &=& f_\zeta \ \ \text{for all $\zeta\in B_1$}, \\  
	j^k_{Q}(f_{0,\tau}) &=& j^k_{Q}(f_{0,0})\ \ \text{for all $\tau\in B_2$, and}\\
	\frac{\di}{\di\tau}\bigg|_{\tau=0} \Pcal(f_{0,\tau}) &:&
        \C^{N_2} \to \C^{N_2} \ \ \text{is an isomorphism}.
\end{eqnarray*}
The second property is ensured by choosing the deformation 
$f_\tau$ to be fixed to order $k$ at all points of the finite set $Q$, which is 
possible by the proof of \cite[Lemma 3.2.1]{AlarconForstnericLopez2021}. 
Setting $N=N_1+N_2$, it follows that the differential of the map
\[
	B_1\times B_2 \ni (\zeta,\tau) \longmapsto
	\left(j^k_{Q}(f_{\zeta,\tau}),   \Pcal(f_{\zeta,\tau})\right) 
	\in J_{Q,A}^k \times \C^{N_2}  
\]
at the point $(\zeta,\tau)=(0,0)$ is an isomorphism $\C^{N}\to\C^N$.

Choose a compact subset $B$ of $B_1\times B_2$ containing 
the origin in its interior. Since $dh=f_{0,0}\theta$ is homotopic on $V$ to the 
restriction of an algebraic section $\vartheta$ of $\mathcal A$ on $M$,
Theorem \ref{th:ahRunge1} implies that the spray 
$f_{\zeta,\tau}\theta:V\to \Acal$ can be approximated uniformly on 
$L \times B$ by a spray $\tilde f_{\zeta,\tau}\theta :M\to \Acal$
of algebraic sections. 
(Note that the map $\tilde f_{\zeta,\tau}$ is not necessarily algebraic, 
only the product $\tilde f_{\zeta,\tau}\theta$ is algebraic.)
If the approximation is close enough,
the implicit function theorem furnishes a value 
$(\zeta_0,\tau_0)\in B$ close to the origin such that 
\[
	j^k_{Q}(\tilde f_{\zeta_0,\tau_0})= j^k_{Q}(f_{0,0}),\quad
	\Pcal(\tilde f_{\zeta_0,\tau_0}) = \Pcal(f_{0,0}) =0.
\]
It follows that the map $\tilde h:M\to\C^n$ defined by 
\[
	\tilde h(x)=h(q_1)+\int_{q_1}^x \tilde f_{\zeta_0,\tau_0}\theta,
	\quad x\in M,
\]
is an algebraic $A$-immersion which approximates
$h$ uniformly on $K$ and whose $k$-jet agrees with the $k$-jet of $h$ 
at every point of $Q$. However, $\tilde h$ need not have a pole
at every point of $E=\overline M\setminus M=\{x_1,\ldots,x_m\}$, 
so it need not be proper. This can be taken care of as in 
\cite[proof of Theorem 3.1]{AlarconLarusson2023regular};
we recall the main idea.
Recall that $M\setminus L = \bigcup_{j=1}^m D_j\setminus \{x_j\}$. 
For every $j=1,\ldots,m$, choose a point $x'_j \subset D_j \setminus \{x_j\}$
and set $E'=\{x'_1,\ldots,x'_m\}$. Let $\theta_0$ be a nowhere vanishing 
holomorphic $1$-form on a neighbourhood of the compact set 
$D=\bigcup_{j=1}^m \overline D_j\subset \overline M$. 
Choose $v_j\in A$ such that 
\begin{equation}\label{eq:max}
	|v_j| > \sup \bigl\{|f_{\zeta,\eta}(x) \theta(x) /\theta_0(x)| : 
	x\in bD_j,\ (\zeta,\eta)\in B\bigr\}.
\end{equation}
(Note that $f_{\zeta,\eta}\theta/\theta_0$ is a map with values in $A$.) 
Next, pick a small neighbourhood 
$W_j \Subset D_j\setminus \{x_j\}$ of $x'_j$
such that $V\cap W_j=\varnothing$ and set $W=\bigcup_{j=1}^m W_j$.
We extend the spray $f_{\zeta,\eta}$ from $V$ to $V\cup W$ by setting
\begin{equation}\label{eq:ext}
	\text{$f_{\zeta,\eta}=v_j \theta_0/\theta$ on $W_j$ for all 
	$(\zeta,\eta)\in B_1\times B_2$ and $j=1,\ldots,m$.} 
\end{equation}
Since the compact set $L'=L\cup E'$ is Runge in $M$, 
Theorem \ref{th:ahRunge1}
allows us to approximate the spray $f_{\zeta,\eta}\theta:V\cup W\to \Acal$ 
uniformly on $L' \times B$ by a spray $\tilde f_{\zeta,\eta}\theta:M\to \Acal$
of algebraic sections. If the approximation is close
enough, then for every $(\zeta,\eta)\in B$ the section  
$\tilde f_{\zeta,\eta}\theta$ of $\Acal$  
has a pole at every point $x_j\in E$ as follows from 
\eqref{eq:max}, \eqref{eq:ext}, and the maximum principle. 
The proof is concluded as before.
\end{proof}

%
%
\section{Every complete minimal surface of finite total curvature \\ is
isotopic to the real part of a proper algebraic null curve}
\label{sec:flux}

%
%
\begin{proof} [Proof of Theorem \ref{th:flux}]
It suffices to prove the second assertion in the statement.
So let $u_0:M\to\R^n$ $(n\ge 3)$ be a complete nonflat conformal 
minimal immersion of finite total curvature, and let 
$\Fcal_t\in H^1(M,\R^n)$, $t\in [0,1]$, be a homotopy 
such that $\Fcal_0=\Flux_{u_0}$.
We assume without loss of generality that $\Fcal_1=0$.

By  \cite[Theorem 1.1]{AlarconLarusson2022complete}, 
there is a smooth isotopy 
$u_t:M\to\R^n$ $(t\in[0,1])$ of nonflat conformal minimal immersions such 
that $\Flux_{u_t}=\Fcal_t$ for all $t\in[0,1]$; in particular, $\Flux_{u_1}=0$. 
(This is an improved version of \cite[Theorem 1.1]{AlarconForstneric2018Crelle}
with control of the flux along the isotopy.)
Hence, $f_t=\di u_t$ for $t\in [0,1]$ is a smooth homotopy of nonflat 
holomorphic sections of the bundle $\Acal\to M$. 

Pick a point $q\in M$ and a 
family $\Cscr=\{C_1,C_2,\ldots,C_s\}$ of closed oriented curves in $M$
based at $q$ which form a basis of $H_1(M,\Z)$ and such that
$|\Cscr|=\bigcup_{i=1}^s C_i$ is a connected compact Runge set in $M$. 
Let $\Pcal$ denote the associated period map \eqref{eq:periodmap}. 
As in the proof of Theorem \ref{th:AAJI}, we choose a compact set 
$L\subset M$ such that $|\Cscr|\subset L$ and 
$M\setminus L = \bigcup_{j=1}^m D_j\setminus \{x_j\}$, where 
$\overline D_1,\ldots,\overline D_m$ are pairwise disjoint discs in 
$\overline M$ and $E=\{x_1,\ldots,x_m\}=\overline M\setminus M$
is the set of ends of $M$. By \cite[Lemma 3.2.1]{AlarconForstnericLopez2021},
there is a period dominating spray of holomorphic sections 
$f_{\zeta,t}:V\to \Acal|_V$ over a neighbourhood $V\subset M$ of $L$,
depending holomorphically on $\zeta\in B\subset \C^N$, 
where $B$ is a ball around the origin in some Euclidean space $\C^N$,
and smoothly on $t\in [0,1]$, such that $f_{0,t}=f_t|_V$ holds for all $t$. 
Pick a closed ball $B'\subset \C^N$ with $0\in B'\subset B$.
By Theorem \ref{th:ahRunge1}, we can approximate the homotopy
$f_{\zeta,t}$ uniformly on $L\times B' \times [0,1]$ by a homotopy
of algebraic sections $\tilde f_{\zeta,t}:M\to \Acal$ such that 
$\tilde f_{0,0}=f_{0,0}$. Furthermore, by a modification of the device used in 
the proof of Theorem \ref{th:AAJI}, we can ensure that 
$\tilde f_{\zeta,t}$ has a pole at each end of $M$ for every $(\zeta,t)$.
To explain this, pick a nowhere vanishing holomorphic $1$-form
$\theta$ on a neighbourhood of $D=\bigcup_{j=1}^m \overline D_j$
in $\overline M$ and set $D^*=D\setminus\{x_1,\ldots,x_m\}$.
For each $t\in[0,1]$, write $f_t=g_t\theta$, so $g_t:D^*\to\boldA$
is a homotopy of holomorphic maps. For each $j=1,\ldots,m$, 
choose a point $x'_j\in D_j\setminus \{x_j\}$ sufficiently close to $x_j$ 
such that 
\[
	|g_0(x'_j)| > \sup \{|g_0(x)| : x\in bD_j\}.
\]
Such $x'_j$ exists since $f_0$ (and hence $g_0$)
has a pole at $x_j$ for $j=1,\ldots,m$.
By shrinking the neighbourhood $V\supset L$ we can also assume that
$x'_j\notin \overline V$ for $j=1,\ldots,m$. Next, choose a smooth function 
$\xi_j:[0,1]\to [1,\infty)$ such that $\xi_j(0)=1$ and 
\[
	\xi_j(t) |g_0(x'_j)| >\sup \{|g_t(x)| : x\in bD_j\}\ \ \text{for all $t\in[0,1]$.}
\]
By shrinking the ball $0\in B\subset\C^N$ if necessary, it follows that
\begin{equation}\label{eq:xit}
	\xi_j(t) |g_0(x'_j)|  > \sup \{|(f_{\zeta,t}/\theta)(x)| : x\in bD_j\} 
	\ \ 
	\text{for all $\zeta\in B$ and $t\in[0,1]$.} 
\end{equation}
For every $j=1,\ldots,m$, we extend the spray $f_{\zeta,t}$ 
(which is defined on $V$)
to a small open neighbourhood $W_j \Subset D_j\setminus \{x_j\}$ 
of the point $x'_j$, with $\overline V\cap \overline W_j=\varnothing$, by 
\begin{equation}\label{eq:onWj}
	f_{\zeta,t}(x)=\xi_j(t) f_0(x) \ \ 
	\text{for all}\ x\in W_j,\ \zeta\in B,\ \text{and}\ t\in[0,1].
\end{equation}
Note that the extended spray is independent of $\zeta$ on $W_j$,
and $f_{\zeta,0}=f_0$ on $W_j$ for all $\zeta\in B$. 
Set $W=\bigcup_{j=1}^m W_j$, $L'=L\cup\{x'_1,\ldots,x'_m\}$,
and let $0\in B'\subset B$ be a smaller closed ball.
We now approximate the extended spray uniformly on $L'\times B'\times [0,1]$
by a spray of algebraic sections $\tilde f_{\zeta,t}$ of the bundle
$\Acal\to M$ such that $\tilde f_{0,0}=f_0$. If the approximation 
is close, it follows from \eqref{eq:xit}, \eqref{eq:onWj} and
the maximum principle that $\tilde f_{\zeta,t}$ has a pole at 
each end of $M$ for every $\zeta\in B$ and $t\in[0,1]$.

Assuming, as we may, that the approximation of $f_{\zeta,t}$ by 
$\tilde f_{\zeta,t}$ is close enough, the implicit function theorem gives 
a map $[0,1]\ni t\mapsto \zeta(t)\in B'$ with $\zeta(0)=0$
such that $\Pcal(\tilde f_{\zeta(t),t}) =\Pcal(f_{0,t})=\Pcal(f_t)$ holds
for all $t\in[0,1]$. (Note that the values of $\tilde f_{\zeta,t}$
on $W$ do not influence the period map since $|\Cscr|\subset L$.)
It follows that 
\begin{equation}\label{eq:tildeutc01}
	\tilde u_t(x)=u_t(q) + \Re \int_q^x 2\tilde f_{\zeta(t),t},
	\quad x\in M,\ t\in [0,1],
\end{equation}
is a homotopy of complete conformal minimal immersions $M\to \R^n$ 
of finite total curvature such that $\tilde u_0=u_0$,
$\Flux_{\tilde u_t}=\Flux_{u_t}=\Fcal_t$ for all $t\in[0,1]$, and hence 
$\Flux_{\tilde u_1}=0$. Thus, $\tilde u_1$ is the real part of a 
proper algebraic null curve $M\to\C^n$.
\end{proof}

In order to justify the discussion below Theorem \ref{th:flux} concerning the case when the immersion $u_0$ given in the theorem is not complete, we fix $0<c<1$ and modify the above proof by choosing the functions 
$\xi_j:[0,1]\to [1,\infty)$ such that \eqref{eq:xit} holds for all $t\in[c,1]$. 
Following the proof without any further modifications, we conclude that the conformal minimal immersion $\tilde u_t$ in \eqref{eq:tildeutc01} is complete for all $t\in[c,1]$.

%
%
%
%
\section{The h-principle for complete nonflat minimal surfaces \\ of finite total curvature}\label{sec:whe1}

\noindent
Assume that $M$ is an affine Riemann surface, 
$\Ascr_*^1(M,\boldA)$ is the space of nonflat algebraic
1-forms on $M$ with values in the null quadric $\boldA$ \eqref{eq:nullq}, 
and $\Ascr_\infty^1(M,\boldA)$ is the subset of $\Ascr_*^1(M,\boldA)$
consisting of $1$-forms with an effective pole at every end of $M$.
We denote by $\Acal\to M$ the algebraic fibre bundle
whose sections are 1-forms on $M$ with values in $\boldA$.

The main aim of this section is to prove Theorem \ref{th:whe1}. 
To this end, we first prove the following h-principle for the $(1,0)$-differential
$\di$, applied on either of the spaces 
$\Re \NC_*(M,\C^n) \subset \CMI_*(M,\R^n)$
and taking values in $\Ascr_*^1(M,\boldA)$ or $\Ascr_\infty^1(M,\boldA)$.

%
%
\begin{theorem}\label{th:hp1}
Assume that $M$ is an affine Riemann surface. 
Let $Q$ be a closed subspace of a compact Hausdorff space $P$,
$u:M\times Q\to \R^n$, $n\ge 3$, be a continuous map
such that $u_p=u(\cdotp,p) \in \CMI_*(M,\R^n)$ for all $p\in Q$, and 
$\phi:M\times P\to \Acal$ be a continuous map such that 
\begin{enumerate}[\rm (a)]
\item $\phi_p=\phi(\cdotp,p)\in \Ascr_*^1(M,\boldA)$ for every $p\in P$, and 
\item $\di u_p=\phi_p$ for every $p\in Q$.
\end{enumerate}
Then there is a homotopy $\phi^t:M\times P\to \Acal$, $t\in [0,1]$, 
such that $\phi^0=\phi$ and the following conditions hold.  
\begin{enumerate}[\rm (i)]
\item $\phi^t_p=\phi^t(\cdotp,p)\in \Ascr_*^1(M,\boldA)$ for every 
$p\in P$ and $t\in [0,1]$. 
\item $\phi^t_p=\phi_p$ for every $p\in Q$ and $t\in [0,1]$. 
\item  $\phi^1_p\in \Ascr_\infty^1(M,\boldA)$ for every $p\in P$.
\item $\Re \int_C \phi^1_p =0$ for every $p\in P$ and $[C]\in H_1(M,\Z)$.
\end{enumerate}
If {\rm (a)} is replaced by the stronger assumption 
\begin{enumerate}[\rm (a'')]
\item[\rm (a')] 
	$\phi_p=\phi(\cdotp,p)\in \Ascr_\infty^1(M,\boldA)$ for every $p\in P$,
\end{enumerate}
then conditions {\rm (i)} and {\rm (iii)} can be replaced by  
\begin{enumerate}[\rm (i)]
\item[\rm (i')]  $\phi^t_p\in \Ascr_\infty^1(M,\boldA)$ 
for every $p\in P$ and $t\in [0,1]$.
\end{enumerate}
If $u_p\in\Re \NC_*(M,\C^n)$ for all $p\in Q$, then {\rm (iv)} can be replaced by 
\begin{enumerate}[\rm (i)]
\item[\rm (iv')] $\int_C \phi^1_p =0$ for every $p\in P$ and $[C]\in H_1(M,\Z)$.
\end{enumerate}
\end{theorem}

It follows from (iii) and (iv) that the maps $u^1_p:M\to\R^n$,
$p\in P$, defined by
\[
	u^1_p(x) = c_p + \Re \int_{x_0}^x 2\phi^1_p 
	\quad \text{for}\ x\in M,
\]
for a fixed $x_0\in M$ and with suitably chosen constants 
$c_p\in\R^n$, $p\in P$, form a continuous family 
$P\to \CMI_*(M,\R^n)$, $p\mapsto u^1_p$,
of complete nonflat conformal minimal immersions of finite total 
curvature such that $u^1_p=u_p$ for all $p\in Q$. 
(The map $c:P\to\R^n$ can be taken to be any continuous extension 
of the map $Q\to\R^n$, $p\mapsto u_p(x_0)$.)  
In case (iv'), the maps $h^1_p:M\to\R^n$ defined by
\[
	h^1_p(x) = c_p + \int_{x_0}^x \phi^1_p \quad\text{for}\ x\in M,
\]
with suitable constants $c_p\in\C^n$, $p\in P$,
are proper nonflat algebraic null curves.

\begin{proof}
We shall use the notation established in the previous two sections.
Pick a point $q\in M$ and a family $\Cscr=\{C_1,C_2,\ldots,C_s\}$ 
of closed oriented curves in $M$ based at $q$ which form a basis of 
$H_1(M,\Z)$ and such that $|\Cscr|=\bigcup_{i=1}^s C_i$ is a connected 
compact Runge set in $M$. 
Let $\Pcal$ denote the associated period map \eqref{eq:periodmap}. 
Denote by $E=\{x_1,\ldots,x_m\}=\overline M\setminus M$
the set of ends of $M$. Let $L\subset M$ be a compact set such that 
$|\Cscr|\subset L$ and $M\setminus L = \bigcup_{j=1}^m D_j\setminus \{x_j\}$, 
where $D_1,\ldots,D_m$ are discs in $\overline M$ with pairwise disjoint 
closures and $x_j\in D_j$ for $j=1,\ldots,m$. 

Choose a nowhere vanishing holomorphic 1-form $\theta$ on $M$ 
and write $\phi_p=f_p\theta$, where $f_p: M\to\boldA$ is a 
holomorphic map depending continuously on $p\in P$.
By \cite[Theorem 5.3]{ForstnericLarusson2019CAG},
there is a homotopy $f^t : M\times P\to \boldA$ $(t\in [0,1])$ such that 
$f^t_p := f^t(\cdotp,p) : M\to \boldA$ is a nonflat holomorphic map
for every $(p,t)\in P\times [0,1]$ satisfying the following two conditions:
\begin{itemize}
\item[\rm (1)]  $f^t_p =f_p$ \ for every 
$(p,t)\in (P\times \{0\}) \cup (Q\times [0,1])$, and  
\item[\rm (2)]  $\Re \Pcal(f^1_p)=0$ for every $p\in P$.
\end{itemize}
If $u_p\in\Re \NC_*(M,\C^n)$ for all $p\in P$ then 
condition (2) can be replaced by 
\begin{itemize}
\item[\rm (2')]  $\Pcal(f^1_p)=0$ for every $p\in P$.
\end{itemize}
Although \cite[Theorem 5.3]{ForstnericLarusson2019CAG} 
is stated for the case when $Q\subset P$ are Euclidean compacts,
the proof holds for any pair of compacts, just as 
\cite[Theorem 4.1]{ForstnericLarusson2019CAG}.

Pick an open Runge neighbourhood $V\Subset M$ of $L$.
By \cite[Lemma 3.2.1]{AlarconForstnericLopez2021},
there is a $\Pcal$-period dominating spray of holomorphic maps 
$f^t_{\zeta,p}:V\to \boldA$, depending continuously on $(p,t)\in P\times [0,1]$ 
and holomorphically on $\zeta\in B\subset \C^N$, 
where $B$ is a ball around the origin in some $\C^N$,
such that $f^t_{0,p}=f^t_p|_V$ for all $(p,t)\in P\times [0,1]$. 

We now adapt to this situation the device,  
used in the proof of Theorem \ref{th:AAJI}, to ensure effective poles
of the approximating algebraic 1-forms at all ends of $M$.
Pick a nowhere vanishing holomorphic $1$-form
$\theta_0$ on a neighbourhood of the compact set 
$D=\bigcup_{j=1}^m \overline D_j$ in $\overline M$.
Also, choose a closed ball $B'\subset B$ containing the origin 
in its interior. For every $j=1,\ldots,m$, we choose a point 
$x'_j\in D_j\setminus \{x_j\}$ sufficiently close to $x_j$ such that 
$x'_j\notin \overline V$ and for every $p\in Q$ we have 
\begin{equation}\label{eq:xjprime}
	|f_{p}(x'_j)| \,\cdotp |(\theta/\theta_0)(x'_j)| > 
	\sup \{|f^1_{\zeta,p}(x)| \,\cdotp |(\theta /\theta_0)(x)| : 
	x\in bD_j,\ \zeta\in B'\}.
\end{equation}
(Note that $\theta/\theta_0$ is a holomorphic function on 
$D\setminus E\subset M$.) 
Since the algebraic 1-form $\phi_p=f_p \theta$ on $M$  
has a pole at $x_j$ for every $p\in Q$ and $\theta_0$ 
has no zeros on $D$, \eqref{eq:xjprime} holds 
for every $x'_j$ sufficiently close to $x_j$.
We can find for every $j=1,\ldots,m$ a continuous function 
$r_j: P\to [0,+\infty)$ which equals $0$ on $Q$ such that 
\begin{equation}\label{eq:rj}
	(1+r_j(p)) |f_{p}(x'_j)| \,\cdotp |(\theta/\theta_0)(x'_j)| > 
	\sup \{|f^1_{\zeta,p}(x)| \cdotp |(\theta /\theta_0)(x)| : 
	x\in bD_j,\ \zeta\in B'\}
\end{equation}
holds for every $p\in P$. Indeed, by \eqref{eq:xjprime} this holds for $p\in Q$, 
and it then holds for all $p\in P$ if $r_j$ is chosen sufficiently 
large on $P\setminus Q$. We extend the homotopy $f^t_{\zeta,p}$  
to a small disc neighbourhood $W_j\Subset D_j\setminus \{x_j\}$ of $x'_j$,  
with $W_j \cap V=\varnothing$, by setting
\begin{equation}\label{eq:extension}
	f^t_{\zeta,p}(x) = (1+t r_j(p)) f_p(x) \ \ 
	\text{for all $x\in W_j$, $\zeta\in B$, $p\in P$, and $t\in [0,1]$}. 
\end{equation}
Set $W=\bigcup_{j=1}^m W_j$. Thus, the homotopy 
$f^t_{\zeta,p}(x)$ is now defined on the Runge domain 
$V\cup W\subset M$ and is independent of $\zeta$ for $x\in W$.
Since the homology basis $\Cscr$ of $M$ is contained in $V$,
the extended homotopy is still period dominating.

Recall that $\Acal\to M$ is an algebraic fibre bundle with 
flexible fibre $\boldA$. By Corollary \ref{cor:flexible}, we can approximate 
the homotopy of holomorphic sections 
$\phi^t_{\zeta,p}=f^t_{\zeta,p}\theta:M\to\Acal$, uniformly 
on the compact set $L'=L\cup\{x'_1,\ldots,x'_m\}\subset V\cup W$ 
and uniformly in the parameters $(\zeta,p,t)\in B'\times P\times [0,1]$,  
by a homotopy of algebraic sections  
\begin{equation}\label{eq:tildephi}
	\tilde \phi^t_{\zeta,p} = \tilde f^t_{\zeta,p}\theta : M\to \Acal
\end{equation}
such that $\tilde \phi^t_{0,p}=\phi^t_{0,p}$ 
holds for all $(p,t)\in (P\times \{0\}) \cup (Q\times [0,1])$.
If the approximation is close enough, 
it follows from \eqref{eq:rj} and \eqref{eq:extension} that
\begin{equation}\label{eq:pole}
	|(\tilde \phi^1_{\zeta,p}/\theta_0) (x'_j)|  > 
	\sup \big\{ |(\tilde \phi^1_{\zeta,p}/\theta_0)(x)| : 
	x\in bD_j,\ \zeta\in B' \big\}
\end{equation}
holds for every $p\in P$ and $j=1,\ldots,m$.
By the maximum principle it follows that the algebraic 1-form 
$\tilde \phi^1_{\zeta,p}$ has a pole at each end of $M$
for every $\zeta\in B'$ and $p\in P$, that is, 
$\tilde \phi^1_{\zeta,p} \in \Ascr_\infty^1(M,\boldA)$.

If we assume that $\phi_p=\phi(\cdotp,p)\in \Ascr_\infty^1(M,\boldA)$  
for every $p\in P$ (see condition (a') in the theorem), 
then the points $x'_j\in D_j\setminus \{x_j\}$ for $j=1,\ldots,m$  
can be chosen such that the inequality \eqref{eq:xjprime} can be replaced 
for any $p\in P$ by 
\[ 
	|f_{p}(x'_j)| \,\cdotp |(\theta/\theta_0)(x'_j)| > 
	\sup \{|f^t_{\zeta,p}(x)| \,\cdotp |(\theta /\theta_0)(x)| : 
	x\in bD_j,\ \zeta\in B',\ t\in [0,1]\}.
\]
Hence, \eqref{eq:rj} holds with the functions $r_j=1$ and
with $f^1_{\zeta,p}(x)$ on the right hand side replaced by $f^t_{\zeta,p}(x)$
for any $t\in [0,1]$. Defining the extension of $f^t_{\zeta,p}(x)$ to
$W$ as in \eqref{eq:extension} with $r_j=1$ for $j=1,\ldots,m$,
the same argument as above gives a homotopy 
$\tilde \phi^t_{\zeta,p}$ as in \eqref{eq:tildephi} such that 
$\tilde \phi^t_{\zeta,p}\in \Ascr_\infty^1(M,\boldA)$
for all values of the parameters. 

Assuming that the approximation of $f^t_{\zeta,p}$ 
by $\tilde f^t_{\zeta,p}$ is close enough, the period domination 
property of the spray $f^1_{\zeta,p}$ and the implicit function theorem 
give a continuous map $\zeta:P\to B'$ with $\zeta(p)=0$
for all $p\in Q$ such that 
\[
	\Re \Pcal(\tilde f^1_{\zeta(p),p}) = 0\ \  
	\text{for all $p\in P$}. 
\]
If (2') holds, we can obtain $\Pcal(\tilde f^1_{\zeta(p),p}) = 0$ 
for all $p\in P$. Then the homotopy 
\[
	\phi^t_p := \tilde f^t_{t\zeta(p),p} \theta \in \Ascr_*^1(M,\boldA)
	\quad \text{for}\ p\in P\ \text{and}\ t\in [0,1]
\]
satisfies the conclusion of the theorem. 
\end{proof}

%
%
The proof of Theorem \ref{th:hp1} also establishes the following result. 
We state it for more general cones $A\subset \C^n_*$, $n\ge 2$. 
We denote by $\Ascr^1_*(M,A)$ the space of nondegenerate 
algebraic $1$-forms with values in $A$, and by $\Ascr^1_\infty(M,A)$ 
the space of those $1$-forms in $\Ascr^1_*(M,A)$ that have a pole at 
every end of $M$.

\begin{proposition}\label{prop:poles}
Let $M$ be an affine Riemann surface 
and $A\subset \C^n_*$ be a smooth connected flexible 
cone \eqref{eq:A} not contained in any hyperplane. 
Then $\Ascr^1_\infty(M,A)$ is an open dense 
subset of $\Ascr^1_*(M,A)$ and the inclusion 
$
	\Ascr^1_\infty(M,A) \hookrightarrow \Ascr^1_*(M,A)
$
is a weak homotopy equivalence.
\end{proposition}

\begin{proof}
By Corollary \ref{cor:flexible}, the proof of Theorem \ref{th:hp1} 
applies to any cone $A$ as in the proposition. Given a compact Hausdorff
space $P$, a closed subspace $Q\subset P$, and a continuous
map $\phi:P\to \Ascr^1_*(M,A)$ such that $\phi(p)\in \Ascr^1_\infty(M,A)$
for $p\in Q$, we have seen that $\phi$ can be approximated 
by continuous maps $\tilde \phi:P\to \Ascr^1_\infty(M,A)$
such that $\tilde \phi(p)=\phi(p)$ for all $p\in Q$. 
This easily implies the proposition.
\end{proof}

We also need the following result, which 
holds in particular for the null quadric.

\begin{proposition}\label{prop:nondegenerate}
Let $M$ be an affine Riemann surface and 
$A\subset \C^n\setminus\{0\}$, $n\ge 2$, be a smooth connected 
flexible cone not contained in any hyperplane. Then 
$\Ascr^1_*(M,A)$ is a dense open subset of $\Ascr^1(M,A)$ 
and the inclusion 
$
	\Ascr^1_*(M,A) \hookrightarrow \Ascr^1(M,A)
$
is a weak homotopy equivalence.
\end{proposition}

\begin{proof}
Let $\Acal \to M$ be the algebraic subbundle of $(T^*M)^{\oplus n}$
determined by the cone $A$. If $TM$ is algebraically trivial, the conclusion 
follows from the algebraic analogue of the 
general position theorem \cite[Theorem 5.4]{ForstnericLarusson2019CAG}.  
The flexibility of $A$ allows us to replace the complete holomorphic vector 
fields, used in \cite[proof of Theorem 5.4]{ForstnericLarusson2019CAG},  
by complete algebraic vector fields on $A$ with algebraic flows (that is, LNDs).  
If $TM$ is not algebraically trivial, choose a Zariski chart $U\subset M$
(the complement of finitely many points in $M$) such that $TU$ 
is algebraically trivial. The previous argument holds over $U$ 
by using LNDs tangent to the fibre $A$ of the trivial bundle
$\Acal|_U\to U$. Each of them can be extended to $\Acal$ by 
multiplication with an algebraic function on $M$ that vanishes to 
a sufficiently high order on the finite set $M\setminus U$. This gives 
LNDs on $\Acal$ tangent to the fibres of the projection
$\Acal\to M$ and vanishing over $M\setminus U$.
By using their flows, the proof can be completed as in 
\cite[proof of Theorem 5.4]{ForstnericLarusson2019CAG}. 
\end{proof}

%
%
\begin{proof}[Proof of Theorem \ref{th:whe1}]  By a standard argument, 
Theorem \ref{th:hp1} implies that the $(1,0)$-differential $\partial$ is a weak 
homotopy equivalence as a map from $\CMI_*(M,\R^n)$ or 
$\Re \NC_*(M,\C^n)$ to $\Ascr_*^1(M,\boldA)$ or its subspace
$\Ascr_\infty^1(M,\boldA)$. Namely, to obtain an 
epimorphism at the level of $\pi_0$, we take $P$ to be a point and $Q$ to 
be empty, and to obtain a monomorphism at the level of $\pi_{k-1}$ and, 
at the same time, an epimorphism at the level of $\pi_k$ for each $k\geq 1$, 
we take $P$ to be the closed unit ball in $\R^k$ and $Q$ to be its boundary,
the unit $(k-1)$-sphere. Since the inclusion 
$\Ascr_*^1(M,\boldA)\hra \Ascr^1(M,\boldA)$ is a weak homotopy equivalence
by Proposition \ref{prop:nondegenerate}, the same conclusion 
holds for the maps $\di:\CMI_*(M,\R^n)\to \Ascr^1(M,\boldA)$ and
$\di:\Re \NC_*(M,\C^n)\to \Ascr^1(M,\boldA)$.

If two of the maps in a commutative diagram of three maps are weak 
homotopy equivalences, then so is the third map. It follows that the inclusion 
$\Re \NC_*(M,\C^n) \hookrightarrow \CMI_*(M,\R^n)$ is a weak homotopy equivalence as well. 
\end{proof}

The inclusion $\Re \NC_*(M,\C^n) \hookrightarrow \CMI_*(M,\R^n)$ 
also satisfies an h-principle, analogous to the corresponding h-principle 
in the holomorphic category \cite[Theorem 4.1]{ForstnericLarusson2019CAG}. 
The basic h-principle for this inclusion is given by Theorem \ref{th:flux}. 
For the parametric case and more general flexible cones, see Theorem 
\ref{th:flux2}. 

We remark that Theorem \ref{th:hp1} can be upgraded with approximation, 
(jet) interpolation, and flux conditions. We will not deal  
with the first two generalisations, which are well understood from
earlier works (see Chapters 3 and 4 in \cite{AlarconForstnericLopez2021}),
but will discuss the third in the following section.

%
%
\section{Further results on the flux map}
\label{sec:flux2}

\noindent
Let $\CMI_{\mathrm{nf}}(M,\R^n)$ denote the space of nonflat 
conformal minimal immersions from an open Riemann surface $M$ 
to $\R^n$, $n\ge 3$. The flux map 
\[
	\Flux: \CMI_{\mathrm{nf}}(M,\R^n)\to H^1(M,\R^n),
\]
given by \eqref{eq:FluxuC}, is a Serre fibration 
by \cite[Theorem 1.4]{AlarconLarusson2022complete}.
This implies that the weak homotopy type of the space of nonflat conformal minimal immersions $M\to\R^n$ with a given flux $\Fcal \in H^1(M,\R^n)$
does not depend on $\Fcal$.

Assume now that $M$ is an affine Riemann surface. 
Given $\Fcal\in H^1(M,\R^n)$, we denote by 
$\CMI_*^\Fcal(M,\R^n)$ the space of complete nonflat conformal minimal 
immersions $u:M\to\R^n$ of finite total curvature with $\Flux_u=\Fcal$. 
In particular, $\CMI_*^0(M,\R^n)=\Re \NC_*(M,\C^n)$. 
A straightforward modification of the proof of Theorem \ref{th:whe1} 
shows that the maps in the diagram
\[
\xymatrix{
	\CMI_*^\Fcal(M,\R^n)  \ar@{^{(}->}[r]^{\iota} \ar[dr]_\di  
					&  \CMI_*(M,\R^n) \ar[d]^\di   \\ 
	  				&   \Ascr^1(M,\boldA) 
}
\]
are weak homotopy equivalences (cf.\ diagram \eqref{eq:diagram}).
Note that Theorem \ref{th:whe1} corresponds to the choice $\Fcal=0$,
but this choice is not important in our arguments. 
The only difference in the proof is that, 
instead of using \cite[Theorem 5.3]{ForstnericLarusson2019CAG} 
to change the fluxes to $0$, we use the more general results in 
\cite{AlarconLarusson2017IJM,AlarconLarusson2022complete} which enable 
us to change the fluxes to any given flux. This implies that the 
weak homotopy type of the space $\CMI_*^\Fcal(M,\R^n)$ is the same 
for all $\Fcal \in H^1(M,\R^n)$. 

We remark that \cite[Theorem 5.3]{ForstnericLarusson2019CAG} can be 
extended as follows. 

\begin{theorem}\label{th:FL19}
Let $M$ be an open Riemann surface, $\theta$ be a nowhere vanishing 
holomorphic $1$-form on $M$, $A\subset \C^n\setminus\{0\}$ be a smooth 
connected Oka cone as in \eqref{eq:A} not contained in any hyperplane, 
$P$ be a compact Hausdorff space, and $Q\subset P$ be a closed subspace. 
Also let $f:M\times P\to A$ and $\Fcal:P\times[0,1]\to H^1(M,\C^n)$ 
be continuous maps such that $f_p=f(\cdot,p):M\to A$ is a nondegenerate 
holomorphic map for every $p\in P$ and $\Fcal(p,t)$ equals the cohomology 
class of $f_p\theta$ for every $(p,t)\in (P\times\{0\})\cup (Q\times [0,1])$. 
Then there is a homotopy $f^t:M\times P\to A$ $(t\in[0,1])$ such that 
$f_p^t=f^t(\cdot,p):M\to A$ is a nondegenerate holomorphic map for every 
$(p,t)\in P\times[0,1]$, $f_p^t=f_p$ for every 
$(p,t)\in (P\times\{0\})\cup (Q\times[0,1])$, and the cohomology class of 
$f_p^t\theta$ equals $\Fcal(p,t)$ for every $(p,t)\in P\times[0,1]$.
\end{theorem}

The advantage of Theorem \ref{th:FL19} 
with respect to \cite[Theorem 5.3]{ForstnericLarusson2019CAG} 
is that it enables one to prescribe the cohomology class of all $1$-forms 
in the family $f_p^t\theta$ for $(p,t)\in P\times[0,1]$, 
and not just of those with $t=1$, and there are 
no restriction on the cohomology classes of the $1$-forms 
$f_p\theta$ for $p\in Q$. The proof is similar to that of 
\cite[Theorem 5.3]{ForstnericLarusson2019CAG} but uses the full strength 
of \cite[Lemma 3.1]{ForstnericLarusson2019CAG}, following the ideas 
in \cite[proofs of Lemma 3.1 and Theorem 1.1]{AlarconLarusson2022complete}. 
Theorem \ref{th:FL19} can also be improved by adding an approximation 
condition as in \cite[Theorem 5.3(2)]{ForstnericLarusson2019CAG}, 
but we do not need it for the applications in what follows. 
We leave the details to the reader.

Replacing \cite[Theorem 5.3]{ForstnericLarusson2019CAG} by 
Theorem \ref{th:FL19} in the proof of Theorem \ref{th:hp1} 
gives the following h-principle for cohomology classes of nonflat 
algebraic $1$-forms on an affine Riemann surface 
with values in a flexible cone $A\subset \C^n_*$. 
It is an extension of Theorem \ref{th:hp1}, as well as an algebraic 
analogue of Theorem \ref{th:FL19} itself.
Recall that $\Ascr^1_\infty(M,A)$ denotes the space of nondegenerate 
algebraic $1$-forms with values in $A$ and with an effective pole at every
end of $M$.

\begin{theorem}\label{th:flux2}
Assume that $M$ is an affine Riemann surface and 
$A\subset \C^n_*$, $n\geq 2$, is a smooth connected flexible cone, not 
contained in any hyperplane. Let $P$ be a compact Hausdorff space and 
$Q\subset P$ be a closed subspace, and let $\phi:P\to \Ascr_*^1(M,A)$ 
and $\Fcal:P\times[0,1]\to H^1(M,\C^n)$ be continuous maps such that 
$\Fcal(p,t)$ equals the cohomology class of $\phi(p)$  
for all $(p,t) \in (P\times\{0\})\cup (Q\times [0,1])$. 
Then there is a continuous map 
$\Phi:P\times[0,1]\to \Ascr_*^1(M,A)$ such that $\Phi(p,t)=\phi(p)$ for all 
$(p,t)\in (P\times\{0\})\cup(Q\times[0,1])$ and the cohomology class of 
$\Phi(p,t)$ equals $\Fcal(p,t)$ for all $(p,t)\in P\times[0,1]$.
Furthermore, if $\phi(p)\in \Ascr^1_\infty(M,A)$ for all $p\in Q$,  
then $\Phi$ can be chosen such that, in addition to the above, 
$\Phi(p,1)\in \Ascr^1_\infty(M,A)$ for all $p\in P$.
If $\phi(p)\in \Ascr^1_\infty(M,A)$ for all $p\in P$,  
then $\Phi$ can be chosen such that $\Phi(p,t)\in \Ascr^1_\infty(M,A)$ 
for all $p\in P$ and $t\in[0,1]$.
\end{theorem}

The following is an immediate corollary.
 
\begin{corollary}\label{cor:flux1}
If $M$ and $A$ are as in Theorem \ref{th:flux2}, then the map
$\Fcal:\Ascr_*^1(M,A)\to H^1(M,\C^n)$ taking $\phi\in \Ascr_*^1(M,A)$
to its cohomology class is a Serre fibration. 
The same holds for the map $\Fcal:\Ascr_\infty^1(M,A)\to H^1(M,\C^n)$. 
\end{corollary}

Note that the flux homomorphism 
$\Flux_u \in H^1(M,\R^n)$ of a conformal minimal 
immersion $u:M\to\R^n$ is the imaginary component of the period 
homomorphism $\Fcal_u \in H^1(M,\C^n)$ of the 1-form $\di u$ 
with vanishing real part, $\Re \Fcal_u =0$. Hence,
the following is a corollary to Theorem \ref{th:flux2} by 
following \cite[proof of Theorem 1.4]{AlarconLarusson2022complete}.

\begin{corollary}\label{cor:flux2}
If $M$ is an affine Riemann surface, then the flux map 
\[
	\Flux: \CMI_*(M,\R^n)\to H^1(M,\R^n)
\]
is a Serre fibration. 
\end{corollary}

In the last part of this section, we give some further
remarks on completeness of nonflat conformal minimal immersions
of finite total curvature and of directed algebraic immersions.

Let $M$ be an open Riemann surface and $n\ge 3$. Denote by 
$\CMI_{\mathrm{nf}}(M,\R^n)$ the space of nonflat
conformal minimal immersions $M\to \R^n$ and by 
$\CMI_{\mathrm{nf}}^{\mathrm{c}}(M,\R^n)$ its subspace 
of complete immersions. By 
\cite[Corollary 1.2(a)]{AlarconLarusson2022complete} 
(see also \cite{AlarconLarusson2017IJM}), the inclusion 
$\CMI_{\mathrm{nf}}^{\mathrm{c}}(M,\R^n)\hra \CMI_{\mathrm{nf}}(M,\R^n)$
is a weak homotopy equivalence, and is a homotopy equivalence if $M$ 
has finite topological type.

Assume now that $M$ is an affine Riemann surface and denote by 
$\CMI_0(M,\R^n)$ the space of nonflat conformal minimal immersions 
of finite total curvature $M\to\R^n$ (including the incomplete ones). 
Consider the following diagram extending \eqref{eq:diagram}:
\begin{equation}\label{eq:diagram-extended}
\xymatrix{
	\Re \NC_*(M,\C^n)  \ar@{^{(}->}[r]^\iota \ar[dr]_\di & 
	\CMI_*(M,\R^n)  \ar@{^{(}->}[r]^{\iota_0} \ar[d]_\di  
					&  \CMI_0(M,\R^n) \ar[d]^\di & 
					  \\ 
	& \Ascr^1_\infty(M,\boldA) \ar@{^{(}->}[r] 
	&   \Ascr^1_*(M,\boldA) \ar@{^{(}->}[r] & \Ascr^1(M,\boldA)
	}
\end{equation}

\vspace{2mm}

\begin{corollary}\label{cor:whe-extended}
If $M$ is an affine Riemann surface,
then all the maps in the diagram \eqref{eq:diagram-extended} 
are weak homotopy equivalences.
\end{corollary}

\begin{proof}
The inclusion $\iota$ is a weak homotopy equivalence by 
Theorem \ref{th:whe1}. Propositions \ref{prop:poles} and 
\ref{prop:nondegenerate} show that the inclusions in the bottom row 
are weak homotopy equivalences.
By Theorem \ref{th:whe1}, the map $\di$ from either $\Re \NC_*(M,\C^n)$
or $\CMI_*(M,\R^n)$ to $\Ascr^1(M,\boldA)$ is a weak homotopy equivalence.
It follows that the two left-hand side vertical maps $\di$ are weak homotopy
equivalences. (This was already pointed out below equation 
\eqref{eq:inclusions}.) The proof of Theorem \ref{th:hp1} 
(ignoring completeness of the given immersions $u_p$ for $p\in Q$ 
and ignoring conditions (iii) and (iv) in the theorem) shows that the map 
$\di$ in the third column is a weak homotopy equivalence as well. 
Finally, $\iota_0$ is a weak homotopy equivalence by 
Theorem \ref{th:flux2}, although this already follows from the diagram. 
\end{proof}

Assume now that 
$A\subset \C^n\setminus\{0\}$ is a smooth connected flexible cone 
\eqref{eq:A} not contained in any hyperplane. 
Denote by $\Iscr_*(M,A)$ the space of nondegenerate proper algebraic 
$A$-immersions $M\to\C^n$ and by $\Iscr_0(M,A)$ the bigger space of 
nondegenerate algebraic $A$-immersions $M\to\C^n$, 
including the nonproper ones. Arguing as above, 
we obtain the following corollary.

\begin{corollary}
Under the above assumptions, the maps in the diagram
\[
\xymatrix{
	\Iscr_*(M,A)  \ar@{^{(}->}[r] \ar[d]_\di  
					&  \Iscr_0(M,A) \ar[d]^\di 
					  \\ 
	 \Ascr^1_\infty(M,A) \ar@{^{(}->}[r]  &   \Ascr^1_*(M,A) 
}
\]
are weak homotopy equivalences.
\end{corollary}

%
%
\section{Algebraic homotopies and algebraically elliptic cones}
\label{sec:A-homotopies}

\noindent
As mentioned in the introduction, the recent 
\cite[Theorem 1.3]{ArzhantsevKalimanZaidenberg2024} provides a good 
sufficient condition for the punctured cone $A\subset\C_*^n$, $n\geq 2$, on a 
connected submanifold $Y$ of $\P^{n-1}$ with $\dim Y\geq 1$ to be 
algebraically elliptic, namely that $Y$ is uniformly rational.  By contrast, the 
collection of known flexible cones is very small and consists of a few classes of 
examples (luckily, the null quadric $\boldA$ is among them). 
This is no surprise: it is more difficult to construct algebraic flows than 
dominating families of algebraic sprays.  
The main classes of examples consist of the cones $A$ 
that are defined by a flag manifold or a toric manifold 
$Y$ \cite{ArzhantsevZaidenbergKuymzhiyan2012}. 
Other examples are given in \cite{Perepechko2013,MichalekPerepechkoSuss2018,ProkhorovZaidenberg2023}.

In view of this, it is of interest to explore to what extent the flexibility assumption 
on the cone $A$ in Theorem \ref{th:whe2} can be relaxed to $A$ being
algebraically elliptic. 
We are unable to prove Theorem \ref{th:whe2} under this weaker assumption,
but partial results can be obtained. First, Theorem \ref{th:AAJI} 
(or \cite[Theorem 1.1]{AlarconLarusson2023regular}) shows that for every 
algebraically elliptic connected smooth cone $A\subset\C_*^n$, $n\geq 2$, 
which is not contained
in any hyperplane, the map $\Iscr_*(M,A) \to \Ascr^1(M,A)$, 
$h\mapsto dh$, in \eqref{eq:inclusion}
induces a surjection of path components of the two spaces.
The next important question is whether this map is also injective
on path components. We provide a partial answer 
by introducing the following notion of algebraic homotopy.

%
%
\begin{definition}\label{def:a-homotopy}
Let $\pi:\Acal \to M$ be an algebraic map onto an affine algebraic
variety $M$. Algebraic sections $\phi_0, \phi_1:M\to \Acal$ of $\pi$ 
are {\em algebraically homotopic}, abbreviated {\em a-homotopic}, 
if there is an algebraic map $\phi:M\times \C\to \Acal$ 
such that $\phi(\cdotp,t):M\to\Acal$ is a section of $\pi$
for every $t\in \C$ and $\phi(\cdot, t)=\phi_t$ for $t=0,1$.
\end{definition}

Using this new notion, we have the following result, 
asserting that a pair of proper nondegenerate $A$-immersions
with a-homotopic differentials lie in the same connected  
component of the space $\Iscr_*(M,A)$. 

\begin{theorem} \label{th:a-homotopy}
Let $M$ be an affine Riemann surface, $A\subset\C_*^n$, $n\ge 2$, 
be a smooth connected algebraically elliptic punctured cone 
not contained in any hyperplane, 
$\Acal\subset (T^*M)^{\oplus n}$ be the algebraic fibre bundle over $M$ 
with fibre $A$, and $h_0, h_1:M\to\C^n$ be proper nondegenerate  
algebraic $A$-immersions. If the differentials $dh_0$ and $dh_1$ are 
a-homotopic sections of $\Acal$,
then there is a homotopy of proper nondegenerate algebraic $A$-immersions 
$h_t:M \to \C^n$, $t\in [0,1]$, joining $h_0$ and $h_1$.
\end{theorem}

\begin{proof}
The assumptions imply that there is an algebraic map 
$\phi:M\times \C \to \Acal$ such that 
$\phi_p=\phi(\cdotp,p):M\to\Acal$ is a section for every $p\in \C$
and $\phi_p=dh_p$ for $p=0,1$. Since $h_0$ and $h_1$ are nondegenerate,  
$\phi_p$ is nondegenerate for all but finitely many values of 
$p\in\C$. Hence, there is an embedded
real analytic arc $P \subset \C$ connecting $0$ and $1$ such that
$\phi_p$ is nondegenerate for every $p\in P$. 
Let $Q=\{0,1\}\subset P$ and let $\theta$ be a nowhere vanishing holomorphic 
1-form on $M$. By \cite[Theorem 5.3]{ForstnericLarusson2019CAG},
there is a homotopy $\phi_p^t=f_p^t \theta$ $(p\in P,\ t\in [0,1])$ 
of nondegenerate holomorphic sections of the bundle $\Acal\to M$ 
(that is, $f_p^t:M\to A$ is a nondegenerate holomorphic map 
for each $(p,t)$) such that
\begin{enumerate}[\rm (i)] 
\item $\phi_p^t=\phi_p$ for every $(p,t)\in (P\times \{0\})\cup (Q\times [0,1])$,
and
\smallskip
\item $\phi_p^1$ has vanishing periods on all closed curves in $M$
for every $p\in P$.
\end{enumerate}
Choose a basis $\Cscr=\{C_1,\ldots,C_s\}$ of $H_1(M,\Z)$
and a compact set $L\subset M$
containing $|\Cscr|=\bigcup_{i=1}^s C_i$ as in the proof of
Theorem \ref{th:flux}, so that $M\setminus L$ is a union 
of pairwise disjoint punctured discs around the ends of $M$.
By \cite[Lemma 3.2.1]{AlarconForstnericLopez2021}, 
we can find a $\Cscr$-period dominating spray of nondegenerate
holomorphic sections $\phi_{\zeta,p}^t  : V\to \Acal|_V$ over
a neighbourhood $V\subset M$ of $L$, depending holomorphically
on $\zeta\in B\subset\C^N$, where $B$ is a ball around the origin in $\C^N$,
and continuously on the parameters $(p,t)\in P\times [0,1]$, 
such that $\phi_{0,p}^t=\phi_{p}^t$ for all $(p,t)\in P\times [0,1]$.

Choose a smaller ball $B'\Subset B$ around the origin.
By Proposition \ref{prop:Mergelyan}, applied to the extended bundle
$Z=\Acal\times\C^N\to X=M\times \C^N$, 
we can approximate $\phi_{\zeta,p}^t$ uniformly on 
$L\times B' \times P\times [0,1]$ by a homotopy
$\wt \phi_{\zeta,p}^t$ of nondegenerate holomorphic sections of 
the bundle $\Acal\to M$ which depends holomorphically on 
$(x,\zeta,p) \in V\times B' \times U$,
where $U\subset \C$ is a neighbourhood of the arc $P$ and $V\supset L$
is a possibly smaller neighbourhood of $L$, such that
$\wt \phi_{0,p}^t=\phi_p^t$ holds for $p=0,1$ and for all $t\in[0,1]$.
Thus, we may view $\wt \phi_{0,\cdotp}^t$ as a homotopy of holomorphic
sections of the fibre bundle $\Acal\times\C\to M\times \C$ 
over the domain $V\times U\subset M\times \C$, which is independent 
of $t\in [0,1]$ over the algebraic submanifold 
$M\times \{0,1\}\subset M\times \C$ intersected with $V\times U$.
After slightly shrinking the ball $B'\subset\C^N$,
Theorem \ref{th:ahRunge1} (see in particular part (d)) 
allows us to approximate the family $\wt \phi_{\zeta,p}^1$ 
uniformly on $L\times B'\times P$ by an algebraic map 
$\psi:M\times \C^N \times \C  \to \Acal$
such that the period dominating spray of algebraic sections 
\[
	\psi_{\zeta,p}=\psi(\cdotp,\zeta,p): M\to\Acal
	\ \ \text{for $p\in P$ and $\zeta\in\C^N$}
\]
agrees with $\phi_p$ for $p=0,1$ and $\zeta=0$.
If the approximation is close enough, the implicit function
theorem gives a smooth map $P\ni p\mapsto \zeta(p)\in B'$
with values close to $0$ such that $\zeta(p)=0$ for $p=0,1$
and the algebraic 1-form $\psi_{\zeta(p),p}$ has vanishing periods
for every $p\in P$. By integrating these forms
with suitable initial values at a point $q\in |\Cscr|$ 
we get a homotopy of nondegenerate algebraic $A$-immersions 
$h_p:M\to \C^n$ $(p\in P)$ connecting $h_0$ and $h_1$.

If in addition $h_0$ and $h_1$ are proper, then the same device 
that was used in the proof of Theorem \ref{th:hp1} can also be used
in this proof to ensure that all $A$-immersions $h_p$ $(p\in P)$ 
in the above family are proper as well. 
\end{proof}

\begin{remark}
The proof applies in the more general case when the given pair of 
immersions in the theorem need not be proper, but then
the intermediate immersions in the homotopy furnished by the theorem need 
not be proper either. Nevertheless, arguing as at the end of 
Section \ref{sec:flux}, we can ensure that the intermediate immersions 
are proper for all parameter values $t\in[c,1-c]$ for any given $c\in(0,\tfrac 1 2)$.
\end{remark}

%
%
%
%
\section{Local contractibility of algebraic mapping spaces}\label{sec:topological}

\noindent
A topological space $X$ is said to be locally contractible if for every point $x\in X$ and every neighbourhood $U$ of $x$, there is a neighbourhood $V\subset U$ of $x$ such that the inclusion $V\hookrightarrow U$ is homotopic to a constant map.  CW complexes and absolute neighbourhood retracts are locally contractible, as are, obviously, locally convex topological vector spaces.
Corollary \ref{cor:flexible} would allow us to determine the weak homotopy type of many algebraic mapping spaces if we knew that they were locally contractible in the compact-open topology.

\begin{theorem}  \label{th:locally-contractible}
Let $X$ be an affine algebraic variety and $Y$ be a flexible algebraic manifold.  Suppose that the space $\mathscr A(X, Y)$ of algebraic maps $X\to Y$ is locally contractible with respect to the compact-open topology.  Then the inclusion $\mathscr A(X, Y) \hookrightarrow \mathscr O(X, Y)$ induces an injection of path components and an isomorphism of homotopy groups in every degree.
\end{theorem}

\begin{remark}
(a)  If $Y$ is the prototypical flexible algebraic manifold $\C^n$, $n\geq 1$, then, as a topological vector space, $\mathscr A(X, Y)$ is locally convex and hence locally contractible.  We do not at present know how to prove or disprove local contractibility for more general manifolds.

(b)  If $X$ is an affine algebraic manifold and $Y$ is a flexible algebraic manifold, the inclusion $\mathscr A(X, Y) \hookrightarrow \mathscr O(X, Y)$ may or may not induce a surjection of path components.  For example, let $\Sigma^n$ denote the complex $n$-sphere $\{(z_0,\ldots,z_n)\in\C^{n+1}:z_0^2+\cdots+z_n^2=1\}$.  As a smooth affine algebraic variety, $\Sigma^n$ is flexible for all $n\geq 2$ (see \cite[Proposition 1.15.3]{AlarconForstnericLopez2021}).  Loday \cite{Loday1973} 
showed that for $p$ and $q$ odd, every algebraic map $\Sigma^p\times\Sigma^q \to\Sigma^{p+q}$ is null-homotopic, whereas there are non-null-homotopic continuous maps $\Sigma^p\times\Sigma^q \to\Sigma^{p+q}$ (see also \cite[Example 6.15.7]{Forstneric2017E}).  On the other hand, it follows from the results in \cite{AlarconLarusson2023regular} that $\mathscr A(X, Y) \hookrightarrow \mathscr O(X, Y)$ does induce a surjection of path components if $X$ is an affine Riemann surface and $Y$ belongs to a class of punctured cones that includes the large cone $\C_*^n$, $n\geq 2$, and the null quadric $\boldA\subset \C_*^n$, $n\geq 3$.

(c)  Under the assumptions of the theorem, $\mathscr O(X, Y)$ and $\mathscr C(X, Y)$ are absolute neighbourhood retracts and hence locally contractible \cite[Proposition 7 and Theorem 9]{Larusson2015PAMS}.  Also, 
$X$ is Stein and $Y$ is algebraically elliptic and hence Oka, so the inclusion $\mathscr O(X, Y) \hookrightarrow \mathscr C(X, Y)$ is a weak homotopy equivalence.
\end{remark}

\begin{proof}[Proof of Theorem \ref{th:locally-contractible}]
We shall prove, at the same time, that the inclusion 
$\mathscr A(X, Y) \hookrightarrow \mathscr O(X, Y)$ induces a 
monomorphism at the level of $\pi_{k-1}$ and an epimorphism 
at the level of $\pi_k$ for all $k\geq 2$.  

Let $B$ be the closed unit ball in $\mathbb R^k$ with boundary sphere $S$.  
Choose a base point $b\in S$.  Let $f:B\to\mathscr O(X,Y)$ be continuous with 
$f(S)\subset\mathscr A(X,Y)$.  What is required is to deform $f$, keeping $f(b)$ 
fixed and keeping $f(S)$ in $\mathscr A(X,Y)$, until all of $f(B)$ lies in 
$\mathscr A(X,Y)$.  We will in fact prove a bit more.

View $B$ as $S\times[0,1]$ with $S\times\{1\}$ identified to the origin $o$ and view $f$ as a continuous map $g:S\times[0,1]\to\mathscr O(X,Y)$, constant on $S\times\{1\}$ and taking $S\times\{0\}$ into $\mathscr A(X,Y)$.  By the algebraic homotopy approximation theorem (the original version \cite[Theorem 3.3]{Forstneric2006AJM} suffices), $\mathscr A(X,Y)$ is dense in any path component of $\mathscr O(X,Y)$ that it intersects.  Also, $\mathscr O(X,Y)$, being an absolute neighbourhood retract, is locally path connected.  Hence, $f(o)\in\mathscr O(X,Y)$ can be deformed to an algebraic map (in fact by an arbitrarily small deformation), so we may assume that $f(o)\in\mathscr A(X,Y)$.

By Corollary \ref{cor:flexible}, 
applied with the parameter spaces $P=S$ and $Q=\varnothing$,
$g$ can be approximated on $S\times[0,1]$ by continuous maps $G:S\times\C\to \mathscr A(X, Y)$ with $G=g$ on $S\times\{0\}$.  Since $\mathscr A(X,Y)$ is locally contractible by assumption, if the approximation is close enough, $G\vert_{S\times\{1\}}$ is homotopic to the constant map with value $f(o)$ through maps with values in an arbitrarily small neighbourhood of $f(o)$ in $\mathscr A(X,Y)$.  

This shows that $f$ may be approximated arbitrarily closely by maps 
$F:B\to\mathscr A(X,Y)$ that equal $f$ on $S$.  Since, again, $\mathscr O(X,Y)$ 
is an absolute neighbourhood retract, if the approximation is close enough, 
then $f$ and $F$ are homotopic as maps $B\to\mathscr O(X,Y)$ by a 
homotopy that is constant on $S$ \cite[Theorem IV.1.2]{Hu1965}.
\end{proof}

Now let $M$ be an affine Riemann surface and $A\subset\C_*^n$, $n\geq 2$, be a connected flexible punctured cone, not contained in a hyperplane, for example the large cone $\C_*^n$ itself or the null quadric $\boldA\subset \C_*^n$, $n\geq 3$.  Let $\mathcal A$ be the algebraic fibre bundle with fibre $A$ over $M$ whose sections are $A$-valued 1-forms.  As before, we write $\mathscr A^1(M,A)$ for the space of algebraic sections of $\mathcal A$ over $M$ and $\mathscr O^1(M,A)$ for the space of holomorphic sections, both endowed with the compact-open topology.

\begin{theorem}  \label{th:locally-contractible-2}
With assumptions as above, suppose that the space $\mathscr A^1(M, A)$ is locally contractible.  Then the inclusion $\mathscr A^1(M, A) \hookrightarrow \mathscr O^1(M, A)$ induces an injection of path components and an isomorphism of homotopy groups in every degree.  If $A=\C_*^n$ or $A=\boldA$, then the inclusion is a weak homotopy equivalence.
\end{theorem}

\begin{proof}
The proof of the first statement is identical to that of Theorem \ref{th:locally-contractible} once we note that since $TM$ is holomorphically trivial, the space $\mathscr O^1(M, A)$ is homeomorphic to the space $\mathscr O(M, A)$ and is therefore an absolute neighbourhood retract.  When $A=\C_*^n$, the inclusion induces a surjection of path components by  \cite[Corollary 2.2]{AlarconLarusson2023regular}.  For $A=\boldA$, this holds by \cite[Proposition 2.3]{AlarconLarusson2023regular}.
\end{proof}


\subsection*{Acknowledgements}
Alarc\'on was partially supported by the State Research Agency (AEI) 
via the grant no.\ PID2020-117868GB-I00, and the \lq\lq Maria de Maeztu\rq\rq\ Unit of Excellence IMAG, reference CEX2020-001105-M, funded by MCIN/AEI/10.13039/501100011033/. Forstneri\v c is supported by the European Union
(ERC Advanced grant HPDR, 101053085) 
and grants P1-0291, J1-3005, and N1-0237 from ARIS, Republic of Slovenia. 
Much of this work was done during a visit of Forstneri\v c and L\'arusson
to the University of Granada in April 2024, and it was completed
during a visit of Alarc\'on and L\'arusson to the University of Ljubljana
in June 2024. We wish to thank the universities for hospitality and support.




\vspace*{4mm}

\noindent Antonio Alarc\'{o}n

\noindent Departamento de Geometr\'{\i}a y Topolog\'{\i}a e Instituto de Matem\'aticas (IMAG), Universidad de Granada, Campus de Fuentenueva s/n, E--18071 Granada, Spain

\noindent  e-mail: {\tt alarcon@ugr.es}

\vspace*{2mm}
\noindent Franc Forstneri\v c 

\noindent Faculty of Mathematics and Physics, University of Ljubljana, Jadranska 19, SI--1000 Ljubljana, Slovenia 

\noindent 
Institute of Mathematics, Physics and Mechanics, Jadranska 19, SI--1000 Ljubljana, Slovenia 

\noindent e-mail: {\tt franc.forstneric@fmf.uni-lj.si} 

\vspace*{2mm}
\noindent Finnur L\'arusson

\noindent Discipline of Mathematical Sciences, University of Adelaide, Adelaide SA 5005, Australia

\noindent  e-mail: {\tt finnur.larusson@adelaide.edu.au}

\end{document}